\documentclass{article}
\usepackage[utf8]{inputenc}
\input{commandes}

\title{Subordination of discrete snakes}
\author{Antoine Aurillard, Mathieu Mourichoux}
\date{}

\begin{document}

\maketitle

\begin{abstract}
Motivated by applications in random geometry, we investigate the notion of subordination of snakes in the discrete setup. More precisely, given a random walk $W$ indexed by a tree $T$ and with steps in $\{...,-1,0,1\}$, we consider its subordinate tree obtained by contracting every edge of $T$ that does not lead to a new record of the walk $W$. When the underlying tree $T$ is a Bienaymé-Galton-Watson tree, we characterize the distribution of this subordinate tree. In particular, when $T$ has a critical offspring distribution in an $\alpha$-stable domain of attraction with $\alpha\in(1,2]$, and under a light tails assumption on the steps, we prove that the associated subordinate tree is itself a Bienaymé-Galton-Watson tree with an offspring distribution in an $\frac{\alpha+1}{2}$-stable domain of attraction. Along the way, we obtain the asymptotic tail of the maximal displacement of the critical branching random walk $W$ in this stable regime, under minimal assumptions. Finally, we use these results to prove scaling limit statements about the subordinate tree.
\end{abstract}

\section{Introduction}

Consider a finite rooted tree $\mathbf{t}$, and a collection of i.i.d. random variables $(J_e)_e$ which assigns a jump $J_e \in \Z$ to each edge $e$ of $\mathbf{t}$. This defines a random walk $W$ indexed by $\mathbf{t}$, also called a $\mathbf{t}$-walk or a $\mathbf{t}$-snake. Formally, $W(u)=\sum_{e \in [\varnothing,u]} J_{e}$ where $[ \varnothing,u]$ is the set of edges in the path from $\varnothing$ to $u$.

In this setting, the linear time is replaced by a tree, but we may still define \emph{record vertices} of $W$ as those vertices $u \in t$ such that if $v \in \mathbf{t}$ is an ancestor of $u$, denoted by $v \prec u$, then $W(v) < W(u)$. For classical random walks, the record times form a sequence which may be characterized as a simple renewal process, but for $\mathbf{t}$-walks a new question arises: consider the record vertices of $W$, and contract all edges in between to get a new tree $S(W)$ which encodes the genealogical relation of those records, as illustrated by Figure~\ref{fig:ExampleSubordination}. What does $S(W)$ look like ? More precisely, how is it distributed and what can we say about its typical size and shape ?

\begin{figure}
    \centering
    \includegraphics[width=\linewidth]{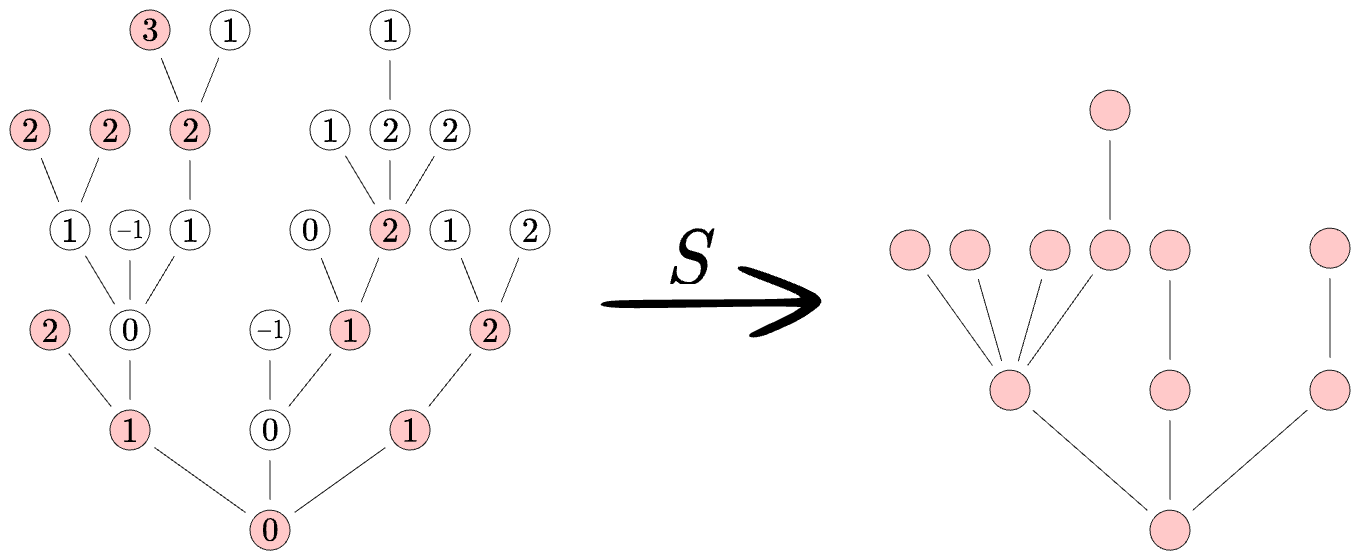}
    \caption{\centering Left: A tree-indexed walk $W$ in $\Z$, with a mark on each of its records. Right: The corresponding subordinate tree $S(W)$.}
    \label{fig:ExampleSubordination}
\end{figure}

This random tree $S:=S(W)$ is called the subordinate tree of the random snake $W$. It is a (measurable) function of $W$, however when the context is clear we may say that $S$ is the subordinate tree of $\mathbf{t}$ and denote it by $S(\mathbf{t})$. The notion of subordination of trees and snakes has been first formally investigated by Le Gall \cite{LeGall2017SubordinationTreesBrownian} in a continuous framework, motivated by application to random geometry (see also \cite{RieraRosalezexcursions} for more general results about subordination of random continuous trees). A striking result is that starting from a Brownian snake, which may be seen as a Brownian motion indexed by a version of the Brownian continuum random tree (CRT) introduced by Aldous \cite{Aldous1991ContinuumRandomTreea}, this subordination procedure gives another well-known continuum random tree called a $\frac{3}{2}$-stable tree. 

As stable trees, including the Brownian CRT which is the $2$-stable tree, are scaling limits of critical Bienaymé trees\footnote{Bienaymé trees are also known as Galton-Watson trees. See \cite{Addario-BerryBrandenbergerHamdanKerriou2022UniversalHeightWidth} for a quick discussion on those terminologies.} conditioned to be large and enjoy a continuous analogue of the branching property of those discrete trees, we chose to investigate whether similar results hold in the discrete framework of critical Bienaymé trees and branching random walks. Moreover, since the stable framework naturally appears, we directly deal with this more general case. Thus in the rest of this paper, excepted when explicitly said otherwise, we make the following assumptions to ensure that we work with discrete versions of those stable trees:

\begin{itemize}
    \item $\mu$ is a critical offspring distribution on $\N=\{0,1,\ldots\}$ in the domain of attraction of an $\alpha$-stable law for some $\alpha \in (1,2]$. In short, we will write $\mu \in \mathrm{dom}(\alpha)$. This means that for \((X_{i})_{i \geq 1}\) i.i.d. with distribution \(\mu\), there is a slowly varying function\footnote{See Appendix~\ref{sct:regular} or \cite{BinghamGoldieTeugels1987RegularVariation} for a general account on slowly varying functions. One must mainly keep in mind Potter's bounds: $\forall \varepsilon>0$, for $x$ large enough, $x^{-\varepsilon} < \ell_{\mu}(x)<x^{\varepsilon}$. We will also make use of regularly varying functions, asymptotic inversion and the Monotone Density Theorem to manipulate those functions, we refer again to \cite[Sections~1.5,~1.7]{BinghamGoldieTeugels1987RegularVariation}.} \(\ell_{\mu}\) such that
\begin{equation}\label{eq:DefiStableCLT}
{\frac{\sum_{i=1}^n X_{i} - n}{n^{1/\alpha}\ell_{\mu}(n)} \text{converges in distribution as $n\to +\infty$}.}
\end{equation}  

i.e. we have a generalized central limit theorem. We will also use that $\mu$ is characterized by the following expansion of its generating function $G_{\mu}$:
\begin{equation}\label{eq:AssumOffspring}
    G_{\mu}(1-s)-(1-s) = s^{\alpha}L_{\mu}\left(\frac{1}{s} \right),
\end{equation}
where $L_{\mu}$ is a slowly varying function related to but different from \(\ell_{\mu}\) (see Section~\ref{ssct:criticalStableBienaymé} for details).

 \item $T$ is a $\mu$-Bienaymé tree, and $T_{n}$ is a $\mu$-Bienaymé tree conditioned to have $n$ edges. 
 
 \item $\theta$ is a \emph{centered} distribution on $\Z$, called the jump distribution since for each edge $e$ its jump $J_{e}$ will be sampled according to $\theta$. Different tail assumptions will be made on $\theta$ throughout this paper.
\end{itemize}

In this context, $T$-walks are better known as \emph{critical branching random walks, or critical BRW} in short, in reference to the branching property of $T$. In order to study their subordinate trees, we will first discuss one of the most studied aspects of BRW, namely the behaviour of their maximum. More specifically, we obtain the tail of the maximum of critical branching random walk where the offspring satisfies \eqref{eq:AssumOffspring}. This generalizes, with a new proof, a recent result of Zhang \cite{Zhang2026EmptyBallsCritical} and a previous result of Lalley \& Shao \cite{LalleyShao2015MaximalDisplacementCritical}. Thanks to this and a discussion on subordination, we prove that Le Gall's result still holds in this discrete framework and we directly extend it to the case where the initial branching random walk, when conditioned to be large, does not converge toward the Brownian snake but toward a stable snake. We also discuss joint scaling limit results for those snakes together with their subordinate tree, and in the particular case of $T_{n}$-walks we provide an estimate on the expected size of the subordinate tree $\E[S(T_{n})]$. 

Finally, as mentioned above, the subordination of trees was first considered by Le Gall for its application to the study of the Brownian sphere (see \cite{legallsurvey2019} for a survey on Brownian geometry). The construction of this random surface is based on the Brownian snake, and Le Gall \cite{LeGall2017SubordinationTreesBrownian} used the associated subordinate tree to study some of its properties. One may wonder if a similar application can be obtained for random maps, which are discrete random surfaces that have been shown to converge toward the Brownian sphere \cite{convergence,uniqueness}. These questions were actually the initial motivation of this study of the discrete framework, and will be addressed in a forthcoming work \cite{article2}.

\subsection{Tail estimate for the maximum of critical BRW}\label{ssec:TailBRW}

With a critical offspring distribution, a branching random walk $W$ dies out almost surely, and we may study its maximal displacement \(\max W\) as a random variable. Lalley \& Shao  \cite{LalleyShao2015MaximalDisplacementCritical} have already studied its tail \(n \mapsto \P(\max W \geq n)\) when the offspring and jump distributions have sufficiently light tails. By linking the \textit{discrete differential equation} satisfied by this tail to its continuous counterpart, namely the true ODE satisfied in the setting of branching Brownian motion, they obtained the following asymptotic:

\begin{equation}\label{eq:LalleyShao}
{ \mathbb P(\max W \geq n) \sim 6 \frac{\sigma_{\theta}^{2}}{\sigma_{\mu}^{2}} \frac{1}{n^{2}}.
}\end{equation}

Note that Le Gall \& Lin \cite{LeGallLin2016RangeTreeindexedRandom} obtained a closely related result, but in their case they relied on the Brownian snake, which is the universal scaling limit of critical branching walks under those light tails assumptions (see \cite{JansonMarckert2005ConvergenceDiscreteSnakes}). With this approach, it is not surprising to see that \eqref{eq:LalleyShao} barely depends on the precise choice of \(\mu\) and \(\theta\). However, one cannot hope that \eqref{eq:LalleyShao} holds when $\mu$ is attracted to an $\alpha$-stable distribution with $\alpha<2$, first because it implies $\sigma_{\mu}^2=+\infty$ but more importantly, it is known since Duquesne \cite{Duquesne2003LimitTheoremContour} and Marzouk \cite{Marzouk2020ScalingLimitsDiscrete} that in this case the scaling limits of large Bienaymé trees and large branching random walks are different from the universal one mentioned earlier. Recently, Zhang \cite{Zhang2026EmptyBallsCritical} still managed to extend Lalley \& Shao's method to the case where $\mu$ satisfies \eqref{eq:AssumOffspring} with $\alpha \in (1,2)$ but $L_{\mu}$ constant, and obtained that $\mathbb P(\max W \geq n)$ is asymptotic (up to a constant) to $n^{-2/(\alpha-1)}$ in this case. Here, we propose a generalization which relates the maximal displacement of the walk to the height of the underlying tree. Instead of Lalley \& Shao's method, a key ingredient is the scaling limit result of Marzouk \cite{Marzouk2020ScalingLimitsDiscrete}. Before discussing this further, we stress that the relation between maximal displacement and height was already hidden in \eqref{eq:LalleyShao}.  Indeed, as soon as \(\sigma_{\mu}^{2} < +\infty\), we have Kolmogorov's estimate \(\mathbb P(Z_n(T) > 0) \sim \frac{2}{\sigma_{\mu}^{2}} \frac{1}{n,}\) and \eqref{eq:LalleyShao} may be written as

\begin{equation}
{ \mathbb P(\max W \geq \sigma_{\theta} n) \sim 3 \mathbb P(Z_{n^2}(T) > 0).
}\end{equation}

 Moreover, Slack \cite{Slack1968BranchingProcessMean} generalized Kolmogorov's estimate to the $\alpha$-stable setting: in this case $P(Z_n(T) > 0)$ is of polynomial order $n^{-1/(\alpha-1)}$ (see \eqref{eq:SlackResult} for a detailed version). Thanks to this, we are able to extend this relation. To do so, we first have to adapt the assumption on the tail of the jump distribution $\theta$. As discussed after the theorem, the minimal assumption required is the following one.

\begin{assum}\label{assum:JumpTailAssumption}
The jump distribution
\(\theta=\text{Law}(J)\) is such that
\begin{equation*}
{\mathbb P\left( \lvert J \rvert \geq \sqrt{ \frac{n}{n^{1/\alpha}\ell_{\mu}(n)} } \right) = o_{n \to +\infty} \left( \frac{1}{n} \right).
}\end{equation*}
where $\ell_{\mu}$ is the slowly varying function from \eqref{eq:DefiStableCLT}. Equivalently,\footnote{see Lemma~\ref{lem:AsymptoticInversionSnake}} we have
\begin{equation*}
    \P\left( \lvert J \rvert \geq n\right) = o\left(\frac{\P(Z_{n^2}(T)>0)}{n^2}\right).
\end{equation*}
\end{assum}
We stress that Assumption~\ref{assum:JumpTailAssumption} implies in particular that $\theta$ has a finite variance.

\begin{theorem}\label{thm:TailMaximumBRW}

Consider \(W = (W(u))_{u \in T}\) a branching random walk with a
critical offspring distribution \(\mu\) and a centered jump distribution
 \(\theta\) (on \(\mathbb{Z}\)). Assume that $\mu \in \mathrm{dom}(\alpha)$ and that $\theta$ satisfies Assumption~\ref{assum:JumpTailAssumption}. Then we have
\begin{equation*}
{\mathbb P\left(  \frac{1}{\sigma_{\theta}}\max W \geq  n \right) \sim_{n \to +\infty} \left( \frac{\alpha+1}{\alpha-1} \right)^{\frac{1}{\alpha-1}} \mathbb P(Z_{n^{2}}(T) > 0).
}\end{equation*} 
In particular, \(n \mapsto \mathbb P(\max W \geq n)\) is
\(\left( -\frac{2}{\alpha-1} \right)\)-varying. 

\end{theorem}

Note that Assumption~\ref{assum:JumpTailAssumption} is consistent with previous related works: Marzouk \cite{Marzouk2020ScalingLimitsDiscrete} made the exact same assumption to establish the convergence of large branching random walks to a Brownian snake driven by a stable tree. Zhang \cite{Zhang2026EmptyBallsCritical} assumed that $\theta$ has a finite moment of order $\frac{2\alpha}{\alpha-1}$, which is slightly stronger than Assumption~\ref{assum:JumpTailAssumption} in his specific case where $\ell_{\mu}$ is constant. Hou \& al. \cite{HouJiangRenSong2023TailProbabilityMaximal} also gave a result similar to Zhang, but for branching Lévy processes with $\alpha$-stable branching, and they also made some slightly stronger moment assumption.
Moreover, Maillard \cite{Maillard2016MaximumTreeindexedRandom} has shown that when the jump distribution $\theta$ is regular but does not satisfy Assumption~\ref{assum:JumpTailAssumption}, we fall in the so-called big jump regime where the maximum $\max W$ has the same tail as the biggest jump in $W$. In particular, when the tail of $\theta$ is \((-\beta)\)-varying for some
\(\beta  > 1\), Assumption~\ref{assum:JumpTailAssumption} fails for $\beta < \frac{2\alpha}{\alpha-1}$ and  Maillard obtained in that case that \(n \mapsto \mathbb P(\max W \geq n)\) is
\(\left( -\beta/\alpha \right)\)-varying. When
\(\beta\) goes to \(\frac{2\alpha}{\alpha-1}\), this exponent goes to the exponent
\(-\frac{2}{\alpha-1}\) of Theorem~\ref{thm:TailMaximumBRW}.

\subsection{Subordination of (sub-)critical snakes}\label{ssct:SubordinationBienaymeTrees}

We now return to the study of the subordinate tree $S$ associated with a $T$-snake $W$, and we first discuss this in the wider setting where we only assume that the offspring distribution $\mu$ has mean $m \leq 1$ (and $\mu \neq \delta_1$) to ensure that $T$ is finite.

It is not difficult to see that in this general setting, $S$ has the branching property, which gives the following result.

\begin{prop}\label{prop:BienaymePreserved}
When $\mu$ is subcritical or critical, $S$ is a Bienaymé tree.
\end{prop}

In the rest of this paper, $\nu$ denotes the offspring distribution of $S$, i.e. the distribution of $Z_1(S)$. It remains to describe this distribution $\nu$ to fully understand the distribution of $S$. However, it does not seem possible to give an explicit expression for $\nu$. To simplify the study of the record vertices, we make the following assumption on $\theta$ to ensure that the branching random walk increases in a continuous way:

\begin{assum}\label{assum:JumpAtMost1}
$\theta$ is supported on $\{ \ldots, -2,-1,0,1 \}$.
\end{assum}

We stress that this covers the case of simple random walk where $\theta(1)=\theta(-1)=p$ and $\theta(0)=1-2p$. When Assumption~\ref{assum:JumpAtMost1} holds, we can characterize $\nu$ through its generating function $G_{\nu}$.

\begin{theorem}\label{thm:FunctionnalCharacterization}
Set $Y=1-J$, where $J$ is centered and has distribution $\theta$. When $\mu$ is subcritical or critical and Assumption~\ref{assum:JumpAtMost1} holds, $Y$ is non-negative a.s. and $G_{\nu}$ is the unique generating function solution of the functional equation
\begin{equation}\label{eq:FunctionnalCharacterization}
    \forall s \in [0,1], G_{\nu}(s)=G_{\mu}\left( \E[G_{\nu}^{\circ Y}(s)]\right),
\end{equation}
where $G_{\nu}^{\circ Y}(s) = \underbrace{ G_{\nu} \circ \dots \circ G_{\nu} }_{ Y \text{ times} }(s)$, and $G_{\nu}^{\circ 0}(s)=s$.

Moreover, the mean $\E[Z_1(S)]$ of $\nu$ is the unique solution in $[0,1]$ of the equation $x = m G_Y(x)$, where $m$ is the mean of $\mu$ and $G_Y$ is the generating function of $Y$. This entails that $\nu$ is critical if and only $\mu$ is critical.
\end{theorem}

\begin{rmk}[On the supercritical case]
Actually, when $\mu$ is supercritical, it is also possible to define $S$ and see that it is a $\nu$-Bienaymé tree where $\nu$ satisfies \eqref{eq:FunctionnalCharacterization}, but this requires to change somewhat the framework. Indeed, by studying the existence of a ray $\xi \in \partial T$ such that $\forall u \in \xi,\ W(u) \leq 0$, one sees that $\P(Z_1(S)=+\infty) >0$ always holds in this case. This means that $\nu$ is an offspring distribution defined on $\{0,1,2,\ldots,+\infty\}$ and we have to consider non-locally finite trees. We also have to give up the order structure of the tree since when $Z_1(S)=+\infty$, the possible existence of several negative rays as above implies that with a positive probability, one cannot enumerate all of those first records with the natural integers. Because of those technicalities and the fact that we are more interested in critical Bienaymé trees, we will not discuss further this case.
\end{rmk}

The implicit characterization of $G_{\nu}$ given in Theorem~\ref{thm:FunctionnalCharacterization} is sufficient to get that subordination preserves criticality, but we did not manage to extract further information on the asymptotic behaviour of $G_{\nu}$ at $1$ (and thus on the asymptotic behaviour of $\nu$) from this characterization. Nonetheless, focusing on our case of interest, namely when $\mu$ is critical and attracted to an $\alpha$-stable distribution, we obtain the asymptotic of $G_{\nu}$ with another method based on Theorem~\ref{thm:TailMaximumBRW}.

\begin{theorem}\label{thm:NewOffspringStable}
Assume that $\mu \in \mathrm{dom}(\alpha)$. If the centered jump distribution $\theta$ satisfies Assumptions~\ref{assum:JumpTailAssumption} and~\ref{assum:JumpAtMost1} then $\nu \in \mathrm{dom}\left(\frac{\alpha+1}{2}\right)$. More precisely, its generating function $G_{\nu}$ has the following expansion at $1$:
\begin{equation*}
    G_{\nu}(1-s)-(1-s)  \sim_{s \to 0} \frac{2}{\sigma_{\theta}} \sqrt{\frac{s^{\alpha+1}}{\alpha+1}L_{\mu}\left(\frac1s\right)},
\end{equation*}
where $L_{\mu}$ is the slowly varying function from \eqref{eq:AssumOffspring}.
\end{theorem}

In particular, we recover a discrete version of Le Gall's result: when $\mu$ has a finite variance, which means that $\alpha=2$ and a large $\mu$-Bienaymé tree looks like the Brownian tree, the subordinate tree $S$ has an offspring distribution $\nu$ attracted to a $\frac{3}{2}$-stable law, so that it converges to a $\frac{3}{2}$-stable tree when conditioned to be large. Note that our proof in the discrete case is also reminiscent of Le Gall's proof for the subordination of the Brownian snake. Namely, we show that $S$ satisfies the branching property, and then we use that  the height of $S$ is $\max W$, hence its distribution is known. However, in the continuous case, the knowledge of this distribution is sufficient to directly characterize the subordinate tree as a $\frac{3}{2}$-stable tree, while in the discrete case, we simply have the tail of this distribution and we will need some analysis to obtain our result.

Theorem~\ref{thm:NewOffspringStable} also reveals that $\nu$ must have a really heavy tail, in particular its variance is infinite, hence it is not possible to realize all Bienaymé trees as subordinate trees. We do not know if every critical offspring distribution $\nu \in \mathrm{dom}(\frac{\alpha+1}{2})$ may be realized by choosing a well-suited $\mu \in \mathrm{dom}(\alpha)$. Nonetheless, given $G_{\nu}$ and $\theta$, it is clear that $G_{\mu}$ is uniquely determined by Theorem~\ref{thm:FunctionnalCharacterization} when it exists, and for a well-chosen $\nu$, namely a special case of critical \emph{Power Fractional Laws}, one can actually solve this equation with unknown variable $\mu$ to get the following explicit example of couple $(\mu,\nu)$:

\begin{expl}\label{ex:ExampleAndrea}
Consider simple symmetric jumps given by $\theta(1)=\theta(-1)=1/2$, and consider the distribution $\nu$ attracted to a $\frac32$-law given by $\nu(n) = \frac{3}{2} \frac{C_{n}}{4^{n}(n+2)}$, where $C_n$ is the $n$-th Catalan number. Then there is an offspring distribution $\mu$ such that \eqref{eq:FunctionnalCharacterization} holds, i.e. satisfying 
\begin{equation*}
    \forall s \in [0,1], G_{\nu}(s)=G_{\mu}\left( \frac{s+G_{\nu}\circ G_{\nu}(s)}{2}\right),
\end{equation*}
which is given by $\mu(n)=\frac{2}{3} \left( \frac{2}{9} \right)^{n} C_{n}.$
\end{expl}

Finally, we point out that all the results in this section may be transferred to a variant $\widetilde{S}:=\widetilde{S}(W)$ of the subordinate tree $S(W)$. This tree $\widetilde{S}(W)$ is formed by the \emph{weak records} of $W$, namely the points $u \in T$ such that $\forall v \preceq u,\ W(v) \leq W(u)$ as depicted in Figure~\ref{fig:WeakRecords}. It may not be defined as the subordinate tree of $T$ with respect to some non-decreasing function built from $W$, but it still enjoys the branching property and its offspring distribution $\widetilde{\nu}$ may be characterized and studied via $\nu$ in the following way:

\begin{corollary}\label{cor:WeakRecords}

\begin{enumerate}[(i)]
    \item When $\mu$ is subcritical or critical, $\widetilde{S}$ is a Bienaymé tree. We let $\widetilde{\nu}$ denote its offspring distribution.
    \item When $\mu$ is subcritical or critical and the centered jump distribution $\theta = \mathrm{Law}(J)$ satisfies Assumption~\ref{assum:JumpAtMost1}, the generating function $G_{\widetilde{\nu}}$ of $\widetilde{\nu}$ is
    \begin{equation*}
        \forall s \in [0,1], G_{\widetilde{\nu}}(s) = G_{\mu}\bigl( \E[G_{\nu}^{\circ J_{-}}(s)] \bigr).
    \end{equation*}
    In particular, $\widetilde{\nu}$ is critical if and only if $\mu$ is critical, else it is subcritical.
    \item When $\mu \in \mathrm{dom}(\alpha)$ and $\theta$ satisfies Assumptions~\ref{assum:JumpTailAssumption} and~\ref{assum:JumpAtMost1}, then $\widetilde{\nu} \in \mathrm{dom}\left(\frac{\alpha+1}{2}\right)$. More precisely, its generating function has the following expansion at $1$:
    \begin{equation*}
        G_{\widetilde{\nu}}(1-s)-(1-s) \sim_{s \to 0} \theta(1)\times \left(G_{\nu}(1-s)-(1-s)\right)  \sim_{s \to 0} \frac{2\theta(1)}{\sigma_{\theta}} \sqrt{\frac{s^{\alpha+1}}{\alpha+1}L_{\mu}\left(\frac1s\right)}.
    \end{equation*}
\end{enumerate}
 
\end{corollary}

\begin{figure}
    \centering
    \includegraphics[width = \linewidth]{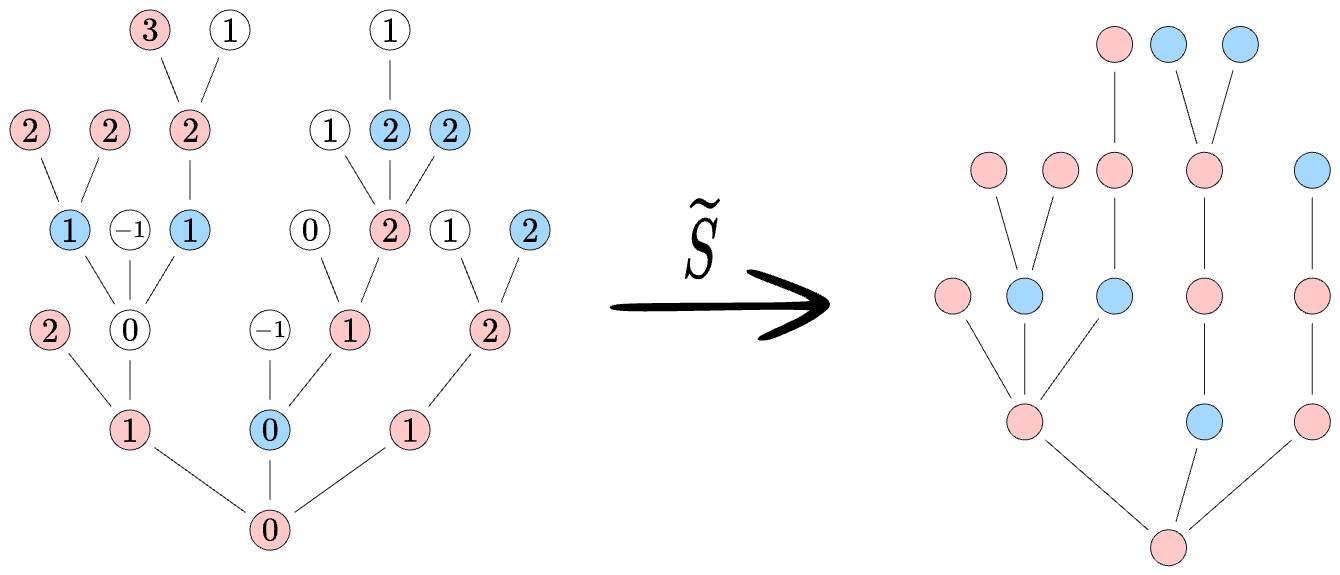}
    \caption{\centering Left: the tree-indexed walk $W$ from Figure~\ref{fig:ExampleSubordination}, with its strict records marked in red and its weak records marked in blue. Right: the associated tree of weak records $\widetilde{S}(W)$, where we kept the colors to emphasize that $S(W)$ is not a subtree of $\widetilde{S}(W)$, instead one must erase the blue vertices in order to recover $S(W)$. }
    \label{fig:WeakRecords}
\end{figure}

\subsection{Scaling limits of large BRW and their subordinate trees}\label{ssct:ScalingLimit}

As a contour function of the subordinate tree $S$ can be expressed by means of the snake $W=(W(u))_{u \in T}$, the existing scaling limit results for random snakes (recalled in Section~\ref{ssct:ScalingLimitsTreesSnakes} and Appendix~\ref{sct:knownScalingLimits}) and a discussion on continuity will imply that those objects converge jointly. However, it does not say much about the limiting distribution of subordinate trees. Theorem~\ref{thm:NewOffspringStable} strongly suggests that $S:=S(T)$ \emph{for a large $T$} converges in distribution to an $\frac{\alpha+1}{2}$-stable tree (conditioned in some way, see again Section~\ref{ssct:ScalingLimitsTreesSnakes} for details). In the particular setting where we condition $\max W$ to be large, such convergence actually is a direct corollary of Theorem~\ref{thm:NewOffspringStable}.

\begin{corollary}\label{cor:ScalingLimitSubordinateCondMax}
Assume that $\mu \in \mathrm{dom}(\alpha)$ and the centered jump distribution $\theta$ satisfies Assumptions~\ref{assum:JumpTailAssumption} and~\ref{assum:JumpAtMost1}. Consider $W_{[n]}$ distributed as a $(\mu,\theta)$-random snake $W$ conditioned on $\max W \geq \sigma_{\theta}n$, together with its subordinate tree $S(W_{[n]})$. We have the following convergence in distribution with respect to the Gromov-Hausdorff topology:
\begin{equation*}
    \frac{1}{\sigma_{\theta}n} S(W_{[n]}) \xrightarrow[n \to +\infty]{(d)} \cT^{[\geq1]}_{\frac{\alpha+1}{2}},
\end{equation*}
where $\cT^{[\geq1]}_{\frac{\alpha+1}{2}}$ is an $\frac{\alpha+1}{2}$-stable tree conditioned to have height at least $1$ as defined in Section~\ref{ssct:ScalingLimitsTreesSnakes}.
Moreover, this convergence holds jointly with the convergence of the encoding processes of $W_{[n]}$ as stated in Proposition~\ref{prop:ConvergenceSnakesCondMax}.
\end{corollary}

We mention that the exponent $\frac{\alpha+1}{2}$ could have be obtained directly from the continuum approach, by adapting the ideas of \cite{LeGall2017SubordinationTreesBrownian} or using the result of \cite[Section 5]{RieraRosalezexcursions}. We also stress that Corollary~\ref{cor:ScalingLimitSubordinateCondMax} relies on the fact that a contour function of $S$ can be expressed as a running maximum of the snake $W=(W(u))_{u \in T}$ along branches (see Section~\ref{ssct:DiscreteSubordination}), but this contour function differs from the standard contour process of $S$ by a random time change (which is also a function of $W$). As explained by Le Gall  \cite{LeGall2017SubordinationTreesBrownian}, understanding this time change and its convergence toward its continuous counterpart (appearing in \cite[Theorem 2]{LeGall2017SubordinationTreesBrownian} in the finite variance case) would give more information, such as joint convergence of all encoding processes and thus convergence with respect to the Gromov-Hausdorff-Prokhorov topology. It would also be a useful step toward obtaining scaling limits in slightly different settings, namely 
\begin{itemize}
    \item for a sequence of i.i.d. random snakes and their subordinate trees;
    \item for a random snake $W_n$ indexed by $T_n$, a $\mu$-Bienaymé tree conditioned to have $n$ edges, and its subordinate tree $S(W_n)$.
\end{itemize}

In this last case, we already know that $T_n$ and $W_n$ converge jointly with $S(W_n)$, again because of this contour function of $S(W_n)$ that may be expressed with $W_n$. However, the dependence between the size of $T_n$ and the shape of $S(W_n)$ is still unclear thus we were not able to give an explicit description of the limit law of $S(W_n)$. Nonetheless, we believe that understanding the distributional limit of $\#S(W_n)$ could lead to this explicit description, and we provide its asymptotic in expectation.

\begin{theorem}\label{thm:expectedNumberOfRecords}
Assume that $\mu \in \mathrm{dom}(\alpha)$ and the jump distribution $\theta$ is centered, has a finite variance and satisfies Assumptions~\ref{assum:JumpAtMost1}. Then, for a $T_n$-indexed random walk $W_n$, i.e. a $(\mu,\theta)$-snake conditioned to have $n$ edges, we have
\begin{equation*}{
\frac{1}{\sigma_{\theta}}\E[\#S(W_{n})] \sim_{n \to +\infty}  \sqrt{ \frac{\alpha}{2\pi} }\frac{\Gamma\left(1-\frac{1}{2\alpha} \right)}{2^{1/\alpha}} \ell_{\mu}(n)^{1/2}n^{\frac{\alpha+1}{2\alpha}},
}\end{equation*}
where $\ell_{\mu}$ is the slowly varying function from \eqref{eq:DefiStableCLT}.

\end{theorem}

The constant $\sqrt{ \frac{\alpha}{2\pi} }\frac{\Gamma\left(1-\frac{1}{2\alpha} \right)}{2^{1/\alpha}}$ is actually linked to the so-called profile of the tree $T_n$, which may be needed to obtain a convergence in distribution for  $\ell_{\mu}(n)^{-1/2}n^{-\frac{\alpha+1}{2\alpha}}\#S(W_n)$. Moreover, this result can also be understood as the asymptotic number of records in $W_n$. Since $\frac{\alpha+1}{2\alpha}$ decreases from $1$ to $\frac{3}{4}$, we see that this number of records heavily depends on the geometry of $T_n$ at large scale, in accordance with the intuition that it must depend on the presence and size of hubs, i.e. nodes with large degree, in the tree. We also see that the number of records for branching random walks is always significantly higher than the approximately $\sqrt n$ records appearing during the first $n$ steps of a standard random walk.

\subsection{Outline}

All the background needed on trees, snakes, their convergence and especially scaling limits of Bienaymé trees and BRW, is given in Section~\ref{sct:preliminaries}. In particular, we give a general definition of subordination in this section. Then, Section~\ref{sct:TailMaxBRW} is devoted to the proof of Theorem~\ref{thm:TailMaximumBRW} on the maximum of BRW, while Section~\ref{sct:SubordinationAndRecords} deals with the characterization of $S$ as a $\nu$-Bienaymé tree via records and the application to weak records. Finally, we prove the scaling limit results on subordinate trees in Section~\ref{sct:ScalinglimitsSubordination}. We also provide Appendix~\ref{sct:knownScalingLimits} to give more details on some adaptions of already existing scaling limits result for snakes, and Appendix~\ref{sct:regular} in case the reader is not familiar with regularly varying functions.

\subsection{Acknowledgement}

We are indebted to Alexis Metz-Donnadieu, who has discussed with us his ongoing work adapting Marzouk's scaling limits of discrete snakes, and Andrea Sportiello, who found the example~\ref{ex:ExampleAndrea}. We also thank Grégory Miermont, Igor Kortchemski, Bruno Schapira, Cyril Marzouk and Benedicte Haas for several helpful discussions.

\section{Preliminaries}\label{sct:preliminaries}

\subsection{Formalism about trees and snakes}\label{ssct:FormalismTreeSnake}

\paragraph{Finite plane trees.}

In this work, we will use the standard formalism for finite plane trees (also known as rooted planar trees) from Neuveu \cite{Neveu1986}, which we quickly recall here. Define
\[\mathcal{U}=\bigcup_{n=0}^\infty (\N^*)^n\]
where $\N^*=\{1,2,3...\}$ and $(\N^*)^0=\{\varnothing\}$. Hence, an element of $\mathcal{U}$ is a finite sequence of positive integers. If $u,v\in\mathcal{U}$, $uv$ stands for the concatenation of $u$ and $v$. If $v$ is of the form $uj$ for some $j\in\N^*$, we say that $u$ is the \emph{parent} of $v$, and that $v$ is a \emph{child} of $u$. Similarly, if $v$ is of the form $uw$ for some $u,w\in\mathcal{U}$, we say that $u$ is an ancestor of $v$ and that $v$ is a descendant of $u$, and we write $u \preceq v$. We define the \emph{generation} of $u$ recursively by $\vert \varnothing \rvert =0$ and for all $j\geq 1, \lvert uj\rvert$ = $\lvert u\rvert+1$. A (finite) rooted planar tree $\mathbf{t}$ is a finite subset of $\mathcal{U}$ such that 
\begin{enumerate}
    \item $\varnothing\in\mathbf{t}$ ($\varnothing$ is called the \emph{root} of $\mathbf{t}$)
    \item if $v\in\mathbf{t}$ and $v\neq\varnothing$, the parent of $v$ belongs to $\mathbf{t}$
    \item for every $u\in\mathcal{U}$, there exists $k_u(\mathbf{t})\geq0$ such that $uj\in\mathbf{t}$ if and only if $j\leq k_u(\mathbf{t})$.
\end{enumerate}
A plane tree can be seen as a graph, where two vertices $u$ and $v$ are linked by an edge if and only if $u$ is the parent of $v$, or vice versa. This graph has a natural embedding in the plane, in which the edges from a vertex $u$ to its children $u1,u2,...,uk_u(\mathbf{t})$ are drawn from left to right. If $u$ and $v$ are two vertices of $\mathbf{t}$, $\llbracket u,v \rrbracket$ denotes the unique simple path of vertices from $u$ to $v$. For every tree $\mathbf{t}$, we denote by $\#\mathbf{t}$ the number of edges of $\mathbf{t}$, $H(\mathbf{t}) = \max_{u \in \mathbf{t}} \vert u \rvert$ its height, and we also set $Z_k(\mathbf{t})=\#\{ u \in \mathbf{t}: \lvert u \rvert =k\}$ for all $k \geq 0$. Moreover, for every vertex $u\in\mathbf{t}$, we let $\mathbf{t}^{(u)}$ be the fringe tree induced by $u$, which is formally defined as
$$v\in\mathbf{t}^{(u)}\text{ if and only if } uv\in\mathbf{t}.$$

Finally, given a tree $\mathbf{t}$ we define its \emph{contour process} $C_{\mathbf t}$ by exploring the tree until all edges have been crossed in both directions and keeping track of the generation at each step. More formally, we define the contour exploration $u_0,u_1,\ldots, u_{2\#\mathbf{t}}$ as follows: 
\begin{itemize}
    \item $u_0 = \varnothing$;
    \item Given $u_0,\ldots,u_i$ for $i < 2\#\mathbf{t}$, $u_{i+1}$ is the first child of $u_i$ that does not belong to $\{ u_0,\ldots,u_{i-1} \}$ if it exists, else $u_{i+1}$ is the parent of $u_i$.
\end{itemize}
Then, the \emph{contour process} $C_{\mathbf t}$ is defined by 
$$\forall k \in \llbracket 0, 2\#\mathbf{t}\rrbracket,\quad C_{\mathbf t}(k)=\lvert u_k \rvert.$$ See Figure~\ref{fig:ContourAndSpatialContour} for an illustration. Note that considering the contour of a finite plane tree is a way to encode it as a Dyck path, and it is well known that this encoding is one-to-one.

\paragraph{Tree-indexed walks, a.k.a. snakes.}

Given a tree $\mathbf{t}$, a random walk $W$ indexed by $\mathbf{t}$, also called a random $\mathbf{t}$-walk or random $\mathbf{t}$-snake depending on context, is defined as a classical random walk but with the time set $\{0,1,2,\ldots\}$ being replaced by $\mathbf{t}$. In this paper we are interested into the $1$-dimensional case of a $\mathbf{t}$-walk $W$ taking value in $\Z$, hence to formally define $W$ we consider a collection of i.i.d. variables $(J_e)_{e \text{ edges in } \mathbf{t}}$ called the \emph{jumps} and distributed according to some distribution $\theta$ on $\Z$. We then set 

\begin{equation*}
    \forall u \in \mathbf{t}, W(u) = \sum_{e \text{ edges in } \llbracket \varnothing,u\rrbracket} J_e.
\end{equation*}

Note that there are more general frameworks for $\mathbf{t}$-snakes, where we may allow dependence between the offspring of a vertex and the jumps from this parent to its children (see e.g. \cite{Marckert2008LineageProcessGaltonWatson}). We will need independence later to study subordinate trees, but those considerations on randomness are not relevant to define subordination itself. In this section, we will thus consider a deterministic $\mathbf{t}$-snake $\mathbf{w}$, which simply is a function $\mathbf{w}: \mathbf{t} \mapsto \Z$. Following the terminology of random walks, we will refer to $\mathbf{w}(u)$ as the \emph{spatial position} of the vertex $u$, but we may also call $\mathbf{w}(u)$ the \emph{label} of $u$ as it is often the case in combinatorics (where $\mathbf{w}$ is rather called a label function on the tree $\mathbf{t}$).

There are essentially two ways to encode a $\mathbf{t}$-snake $\mathbf{w}$ by means of functions. In both cases, we rely on the contour exploration $u_0,u_1,\ldots, u_{2\#\mathbf{t}}$ of $\mathbf{t}$ as defined earlier. 

\begin{enumerate}
    \item Define the \emph{spatial contour process} $C^{sp}_{\mathbf{w}}$ of $\mathbf{w}$ in the same way we have defined the contour process $C_{\mathbf{t}}$ but with the spatial position $\mathbf{w}(u)$ instead of the height $\lvert u\rvert$: $\forall k \in \llbracket 0, 2\#\mathbf{t}\rrbracket, C^{sp}_{\mathbf{w}}(k)=\mathbf{w}(u_k)$. This process alone does not characterize $\mathbf{w}$, but the pair $\left(C_{\mathbf{t}},C^{sp}_{\mathbf{w}}\right)$ does. For this reason, we also call $C_{\mathbf{t}}=C_{\mathbf{w}}$ the contour process of $\mathbf{w}$.
    
\begin{figure}
    \centering
    \includegraphics[width = \linewidth]{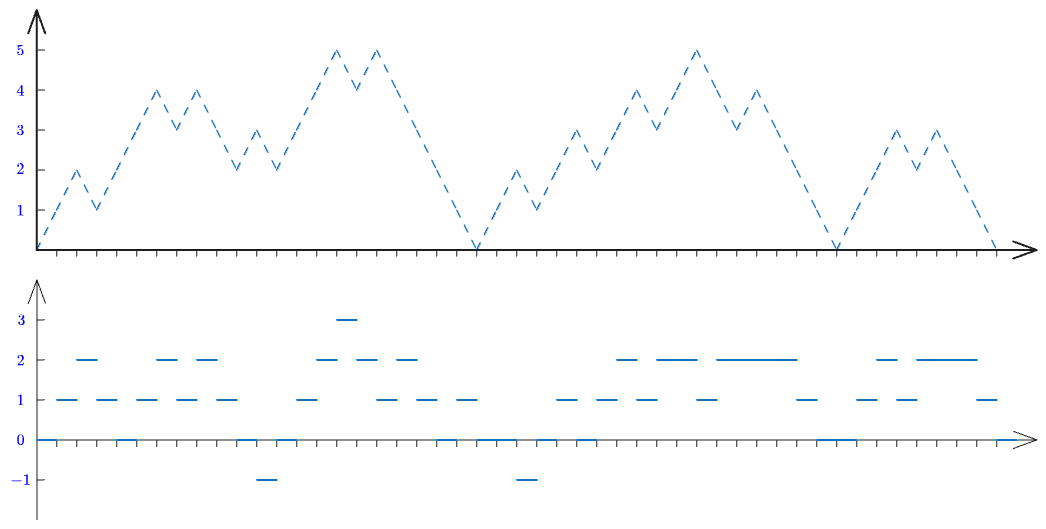}
    \caption{\centering Top: the contour process of the walk $W$ (or its underlying tree) from Figure~\ref{fig:ExampleSubordination}, extended by linear interpolation. Bottom: The spatial contour process of the same walk $W$, without linear interpolation.}
    \label{fig:ContourAndSpatialContour}
\end{figure}
    \item We may also encode $\mathbf{w}$ with a collection of evolving paths (which motivated the name \emph{snake}). More precisely, for all $k \in \llbracket 0, 2\#\mathbf{t}\rrbracket$, consider the random path $\mathbf{w}_{\llbracket \varnothing,u_k\rrbracket}=(\mathbf{w}(v))_{v \in \llbracket \varnothing,u_k\rrbracket}$ as a tuple of length $\lvert u_k\rvert+1 = C_{\mathbf{w}}(k)+1$. Then the collection $\left( \mathbf{w}_{\llbracket \varnothing,u_k\rrbracket} \right)_{0 \leq k \leq 2\#\mathbf{t}}$ characterizes $\mathbf{w}$.
\end{enumerate}

We stress that those two encodings are equivalent. In particular, it is simple to reconstruct $(C_{\mathbf{w}},C_{\mathbf{w}}^{sp})$ from $\left( \mathbf{w}_{\llbracket \varnothing,u_k\rrbracket} \right)_{0 \leq k \leq 2\#\mathbf{t}}$: the contour process is given by the length of those tuples, while the spatial contour process is given by the ending values of those tuples (and is sometimes referred to as the \emph{head of the snake} because of this).

\subsection{Discrete subordination}\label{ssct:DiscreteSubordination}

We now define the notion of subordination in this discrete framework. Following the presentation developed by Le Gall \cite{LeGall2017SubordinationTreesBrownian} for the continuous setting, we first give a general definition and then focus on a specific subordination related to snakes. To this aim, we first introduce a small extension of plane trees that will be convenient here.

\paragraph{Plane trees with edge lengths.}

A plane tree with edge lengths $(\mathbf{t},\delta)$ simply consists of a finite plane tree $\mathbf{t}$ together with a collection $(\delta_e)_{e \text{ edges in } \mathbf{t}}$ of positive numbers. It is naturally equipped with the metric 
$$\delta(u,v)=\sum_{e \text{ edges in } \llbracket u, v\rrbracket} \delta_e,$$
and also has a contour process,  defined from the contour exploration $u_0,u_1,\ldots, u_{2\#\mathbf{t}}$ of $\mathbf{t}$ by setting $C_{(\mathbf{t},\delta)}(k)= \delta(\varnothing,u_k)\ \forall k \in \llbracket 0, 2\#\mathbf{t}\rrbracket$.

Note that the contour process encodes the tree $\mathbf{t}$ (as it is the Dyck path $C_{\mathbf{t}}$ with the vertical lengths distorted) but also the distance $\delta$ in an explicit way: consider the contour exploration $u_0,u_1,\ldots,u_{2\#\mathbf{t}}$ of $\mathbf{t}$, then for all $i<j$ the last common ancestor of $u_i$ and $u_j$ is at distance $\min_{\llbracket i,j \rrbracket} C_{(\mathbf{t},\delta)}$ from the root, hence the distance between $u_i$ and $u_j$ is given by
\begin{equation}\label{eq:DistanceFromContour}
    \delta(u_i,u_j) = C_{(\mathbf{t},\delta)}(i)+C_{(\mathbf{t},\delta)}(j) - 2\min_{\llbracket i,j \rrbracket} C_{(\mathbf{t},\delta)}.
\end{equation}

Finally, we also want to consider $(\mathbf{t},\delta)$ where some edges $e$ have length $\delta_e=0$. However, this means that some vertices $u\neq v \in \mathbf{t}$ may be at distance $0$, in which case we want to identify them by contracting the zero-length edges between them. To do so, we define the \emph{standard contour process} of $(\mathbf{t},\delta)$ from $C_{(\mathbf{t},\delta)}$  by contracting every horizontal segment of its graph (it corresponds to a zero-length edge) to a single point. This standard contour process defines a unique plane tree with positive edge lengths $(\mathbf{t}^*,\delta^*)$, which is isometric to $(\mathbf{t},\delta)$ when seen as metric spaces. By convention, we set  $(\mathbf{t},\delta)=(\mathbf{t}^*,\delta^*)$, and we call $C_{(\mathbf{t},\delta)}$ a non-standard contour process of this tree.

\paragraph{Subordination of trees and snakes.}
Given a tree $\mathbf{t}$, we say that a function 
$g:\mathbf{t}\rightarrow\R$ is \emph{non-decreasing on} $\mathbf{t}$ if 
\begin{equation*}
    g(u)\leq g(v)\quad\text{ for any $u,v\in\mathbf{t}$ with $u\preceq v$.}
\end{equation*}

We may then consider the subordinate tree of $\mathbf{t}$ with respect to any non-decreasing function $g$.

\begin{definition}
The subordinate tree of $\mathbf{t}$ with respect to a non-decreasing function $g$, denoted by $S(\mathbf{t},g)$, is the tree with edge lengths $(\mathbf{t},\Delta g)$ where for all edge $e$ from $u \in \mathbf{t}$ to $ v \in\mathbf{t}$ we set $\Delta g(e) = g(v)-g(u)\geq0$. 
\end{definition}

Keep in mind that when $g(v)=g(u)$, the corresponding edge is contracted into a single vertex. In particular, $g$ becomes a (strictly) increasing function on $S(\mathbf{t},g)$.

Note that $S(\mathbf{t},g)$ has a simple characterization: consider the contour exploration $u_0,u_1,\ldots,u_{2\#\mathbf{t}}$ of $\mathbf{t}$ as defined in Section~\ref{ssct:FormalismTreeSnake}, then  
$$i \in \llbracket 0, 2\#\mathbf{t}\rrbracket \mapsto g(u_i) \text{ is a (possibly non-standard) contour process of } S(\mathbf{t},g).$$

In this paper, we only consider the non-decreasing functions arising as running maximum of tree-indexed walk, defined as follows. Given a tree $\mathbf{t}$, consider a walk $\mathbf{w}$ indexed by $\mathbf{t}$, which is a function $\mathbf{w}:\mathbf{t}\rightarrow\Z$. For every $u\in\mathbf{t}$, we define
\begin{equation}\label{eq:fonction:croissance}
    \overline{\mathbf{w}}(u)=\max_{v\preceq u}\mathbf{w}(v), 
\end{equation}
which clearly is a non-decreasing function on $\mathbf{t}$. The corresponding subordinate tree $S(\mathbf{w}):=S(\mathbf{t},\overline{\mathbf{w}})$ is called the subordinate tree of the snake $\mathbf{w}$. 

According to the previous remark, $S(\mathbf{w})$ can be conveniently expressed with the path representation $\left( \mathbf{w}_{\llbracket \varnothing,u_k\rrbracket} \right)_{0 \leq k \leq 2\#\mathbf{t}}$ of $\mathbf{w}$:
\begin{equation}\label{eq:ContourSubordinate}
    \left( \max\mathbf{w}_{\llbracket \varnothing,u_k\rrbracket} \right)_{0 \leq k \leq 2\#\mathbf{t}} \text{ is a (possibly non-standard) contour process of } S(\mathbf{w}).
\end{equation}

See Figure~\ref{fig:StandardAndNonStandardContourSubordinate} to compare this contour process to the standard one. Also observe that when $v$ is a child of $u$,  $\overline{\mathbf{w}}(v) > \overline{\mathbf{w}}(u)$ means that  $\mathbf{w}(v) > \mathbf{w}(u')$ for any $u' \prec v$, i.e. $v$ is a \emph{record} of $\mathbf{w}$ as defined in Introduction. Hence, all edges that do not lead to a record are contracted in $S(\mathbf{w})$, and $\#S(\mathbf{w})+1$ is the number of records in $\mathbf{w}$.

Finally, in the rest of this paper, we will almost only consider walks $\mathbf{w}$ such that $\mathbf{w}(v)-\mathbf{w}(u)\leq1$ when $v$ is a child of $u$ (due to Assumption~\ref{assum:JumpAtMost1}). In this case we have further simple relations between $S(\mathbf{w})$ and $\mathbf{w}$. Indeed, this implies that $\overline{\mathbf{w}}$'s increments take values in $\{0,1\}$ and $S(\mathbf{w})$ is simply obtained by contracting the edges where $\overline{\mathbf{w}}$ remains constant. We thus have
\begin{equation}\label{eq:generationRecords}
    Z_k\left(S(\mathbf{w})\right)=\#\{u\in\mathbf{t},\text{ $u$ is a record of }\mathbf{w}\text{ and }\mathbf{w}(u)=k\},
\end{equation}
and
\begin{equation}\label{eq:lien:hauteur}
    \max_{u\in\mathbf{t}}\mathbf{w}(u)=H\bigl(S(\mathbf{w})\bigr).
\end{equation}

\begin{figure}
    \centering
    \includegraphics[width = \linewidth]{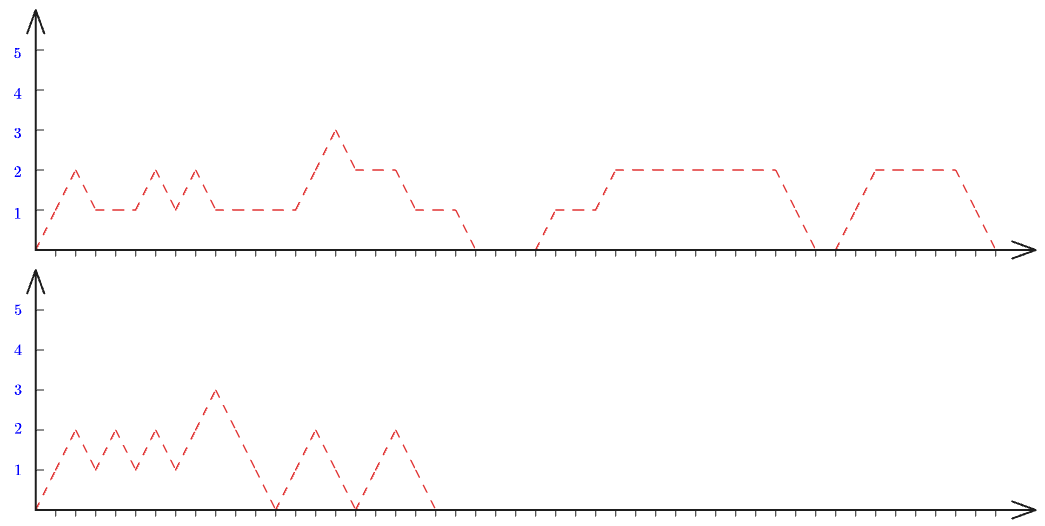}
    \caption{\centering Top: Non-standard contour process of $S(W)$ obtained by following the running maximum of $W$ along its branches. Bottom: Standard contour process of $S(W)$, obtained by contracting all horizontal edges of the non-standard process. In both cases, the tree-indexed walk $W$ is from Figure~\ref{fig:ExampleSubordination} and the contour processes are extended by linear interpolation.}
    \label{fig:StandardAndNonStandardContourSubordinate}
\end{figure}

\subsection{Convergence of trees and snakes}\label{ssct:ConvTreeSnake}

The main point of the encoding processes introduced earlier is that it is sufficient to study their convergence to obtain that the corresponding trees or snakes also converge. We first introduce quickly the notion of convergence of trees before focusing on this sufficient condition (see e.g. \cite{LeGall2005RandomTreesApplications} for a more complete discussion). Then, based on \cite{MarckertMokkademStatesSpacesSnake}, we present the frameworks for the convergence of snakes and explain why it is again sufficient to study the convergence of the contour and spatial contour processes to obtain convergence in a seemingly stronger sense.

\paragraph{Convergence of trees via contour functions.}

A finite tree $\mathbf{t}$ is a graph, hence we may consider its graph distance $d_{\mathbf{t}}$ and see it as a \emph{compact metric space} $(\mathbf{t},d_{\mathbf{t}})$. We may then rescale it by defining $\lambda \mathbf{t}$ as the metric space $(\mathbf{t},\lambda\times d_{\mathbf{t}})$ where all distance are multiplied by $\lambda \in \R_+$. The same holds for finite trees with edge lengths. We further equip the set of (isometry class of) compact metric spaces with the Gromov-Hausdorff distance $d_{GH}$. Thanks to this point of view, we say that a compact metric space $\cT$ is a scaling limit of a sequence of trees $(\mathbf{t}_n)_n$ if there is a sequence of positive numbers $(\lambda_n)_n$ going to $0$ such that 
$$
d_{GH}(\lambda_n \mathbf{t}_n, \cT) \to 0 \text{ as } n \to +\infty.
$$

Actually, not all metric spaces may be realized as a scaling limit of some trees. The possible limits are called $\R$-trees and have several characterizations, but in this paper we will only focus on their definition through a contour function. The main idea is to mimic \eqref{eq:DistanceFromContour}, which tells us how to recover the distance in the tree from a contour, but in a continuous setting. Consider $h \in C_c(\R_+,\R_+)$ such that $h(0)=0$. Then define on its compact support $\mathrm{supp}(h)=[0,\zeta]$ the pseudo-distance
\begin{equation*}
    \forall s <t \in [0,\zeta],\  d_h(s,t) = h(s) +h(t) - 2 \min_{[s,t]}h.
\end{equation*}

We may have $s<t$ such that $d_h(s,t)=0$, but as soon as we consider $\cT_h$ the quotient space of $[0,\zeta]$ by the equivalence relation $s \sim_h t \Leftrightarrow d_h(s,t)=0$, $d_h$ becomes a distance on $\cT_h$ which makes it a compact metric space. We call $(\cT_h,d_h)$ the $\R$-tree associated with the contour function $h$. However, notice that two contour functions $h \neq g$ may encode two isometric $\R$-trees (which we see as the same $\R$-tree). In particular, for all continuous non-decreasing  $\phi:\R_+ \mapsto \R_+$ such that $\phi(0)=0$ and $\max \phi \geq \max \mathrm{supp}(h)$, $h$ and $h\circ\phi$ encodes the same $\R$-tree.

Finally, the map $h \mapsto \cT_h$ is continuous. More precisely, we have
\begin{equation*}
    d_{GH}(\cT_h,\cT_g) \leq 2\lVert h-g\rVert_{\infty}.
\end{equation*}

As a consequence, we may focus on those contour processes to obtain the scaling limit of a sequence of trees $(\mathbf{t}_n)_n$. Indeed, for every $n$ consider a (possibly non-standard) contour process $C_{n}$ of $\mathbf{t}_n$, and extend it by linear interpolation to see it as a continuous function. If one exhibits some positive numbers $(\lambda_n)_n$, some continuous non-decreasing functions $(\phi_n)_n$ from $\R_+$ to $\R_+$ (such that $\phi_n(0)=0$ and $\max \phi_n \geq \mathrm{supp}(C_n)$) and a contour function $h$ such that 
\begin{equation*}
    \lVert \lambda_n C_{n}\circ \phi_n - h \rVert_{\infty} \to 0,
\end{equation*}
then $\lambda_n \mathbf{t}_n \to \cT_h$ with respect to the Gromov-Hausdorff topology.

\paragraph{Equivalent notions of convergence for snakes.}

Both ways to encode a snake $\mathbf{w}$ may be used to study convergence, but Marckert \& Mokkadem \cite{MarckertMokkademStatesSpacesSnake} proved that these two notions of convergence are actually equivalent. Here, we briefly introduce their framework and their homeomorphism theorem with our notation for a later use.

We first discuss the representation of the snake $\mathbf{w}$ with its contour and spatial contour processes $(C_{\mathbf{w}},C_{\mathbf{w}}^{sp})$, also called \emph{the tours} of the snake by Marckert \& Mokkadem. Since this representation extends the coding of a tree by its contour process, we are interested in the convergence of $C_{\mathbf{w}}$ toward a contour function as defined in the previous paragraph, jointly with the convergence of $C_{\mathbf{w}}^{sp}$ toward a continuous function compatible with this contour. Marckert \& Mokkadem asked for those two continuous functions to be defined on $[0,1]$, but here we only require that they are supported on a common segment, say $[0,\zeta]$. Thus, we define the set of limiting objects, namely the \emph{state space of tours} $\T$, as the subset of $C_c(\R_+,\R_+)\times C_c(\R_+,\R)\times \R_+$ of elements $(h,f,\zeta)$ satisfying:
\begin{itemize}
    \item $h$ is supported on $[0,\zeta]$ and $h(0)=0$;
    \item $f$ is supported on $[0,\zeta]$, $f(0)=0$ and for all $s,s' \geq 0,\ d_h(s,s')=0$ implies $f(s)=f(s')$.
\end{itemize}

Notice that, in light of \eqref{eq:DistanceFromContour}, the process $(C_{\mathbf{w}},C_{\mathbf{w}}^{sp},2\#\mathbf{t})$ extended by linear interpolation is in $\T$.
Finally, the distance on $\T$ is given by 
\begin{equation*}
    d_{\T}\bigl((h,f,\zeta),(h',f',\zeta')\bigr) = \lVert h-h' \rVert_{\infty} \vee \lVert f-f' \rVert_{\infty} \vee \lvert \zeta -\zeta'\rvert.
\end{equation*}

For the representation of a snake $W$ as a collection of paths $\left( W_{\llbracket \varnothing,u_k\rrbracket} \right)_{0 \leq k \leq 2\#\mathbf{t}}$, we want to establish convergence toward a continuous collection of continuous paths. Moreover, we want the limiting objects to preserve the following property of a discrete snake: given the vertices $u_i,u_j$ in the contour exploration of $\mathbf{t}$, we have seen that their last common ancestor is at height $\min_{\llbracket i,j\rrbracket} C_{\mathbf{w}}$, hence we must have $W_{\llbracket \varnothing,u_i\rrbracket}(k)=W_{\llbracket \varnothing,u_j\rrbracket}(k)$ for $k \leq \min_{\llbracket i,j\rrbracket} C_{\mathbf{w}}$.

Considering all this, we define the \emph{state space of snakes} $\W$ as the subset of $C_c(\R_+,\R_+)\times C_c\left(\R_+,C(\R_+,\R)\right) \times \R_+$ of elements $(h,\omega,\zeta)$ satisfying:
\begin{itemize}
    \item $h$ and $\omega$ are supported on $[0,\zeta]$, $h(0)=0$ and $\omega(0)$ is the null function;
    \item for all $s \in [0,\zeta]$, $\omega_s:=\omega(s)$ is a path stopped at $h(s)$, i.e. $\forall t \geq h(s), \omega_s(t)=\omega_s(h(s))$;
    \item for all $s \leq s'$, $\omega_s(t)=\omega_{s'}(t)$ for $t \leq \min_{[s,s']} h$.
\end{itemize}

We endow this space with the distance
\begin{equation*}
    d_{\W}\bigl((h,\omega,\zeta),(h',\omega',\zeta')\bigr) = \lVert h-h' \rVert_{\infty} \vee \max_{s,t \in \R_+^2} \lvert \omega_s(t) - \omega'_s(t) \rvert \vee \lvert \zeta -\zeta'\rvert.
\end{equation*}

We can now state the homeomorphism theorem which says that the two frameworks introduced above are equivalent. Thanks to this result, we can focus on the study of contour and spatial contour, and then switch to the representation as a collection of paths when an application requires this stronger framework.

\begin{theorem}[Homeomorphism theorem, {\cite[Theorem~2.1]{MarckertMokkademStatesSpacesSnake}}]\label{thm:homeo}
For all $(h,\omega,\zeta) \in \W$, define $\mathfrak H(h,\omega,\zeta) = (h,f,\zeta)$ where
\begin{equation*}
    \forall s \in [0,\zeta], f(s) = \omega_s(h(s)).
\end{equation*}
Then $\mathfrak H$ is an homeomorphism from $\W$ to $\T$.
\end{theorem}

\begin{rmk}
\cite[Theorem~2.1]{MarckertMokkademStatesSpacesSnake} actually proves this result when all elements in $\T$ and $\W$ must have $\zeta=1$. Here, we add some flexibility for the support $[0,\zeta]$ of those encoding functions, but given our choice of distance $d_{\T}$ and $d_{\W}$ this somewhat more general result is a direct corollary of \cite[Theorem~2.1]{MarckertMokkademStatesSpacesSnake}.
\end{rmk}

\subsection{Critical Bienaymé trees in a stable domain of attraction}\label{ssct:criticalStableBienaymé}

We recall that a $\pi$-Bienaymé tree $T$, where $\pi$ is a probability distribution on the non-negative integers, is a random tree characterized by the so-called \emph{branching property}: The root has a number $Z_1(T)$ of children distributed according to $\pi$, and given $Z_1(T)$ the subtrees $T^{(1)},\ldots,T^{\left(Z_1(T)\right)}$ are i.i.d. copies of $T$. 

Moreover, $T$ is almost surely finite if and only if $\pi$ has a mean $\leq 1$ (and $\pi \neq \delta_1$), which includes the critical case where $\pi$ has mean $1$.

In this paper, we will mostly manipulate  $\pi$-Bienaymé tree where $\pi \in \mathrm{dom}(\beta)$, i.e. $\pi$ critical and in the domain of attraction of a $\beta$-stable distribution, where $\beta \in (1,2]$. We gather in this section the properties of such offspring distributions useful to us. Note that this requires the notions of slowly varying functions and regularly varying functions, whose definitions and useful properties are recalled in Appendix~\ref{sct:regular}.

As explained in introduction, an offspring distribution $\pi$ is such that $\pi \in \mathrm{dom}(\beta)$ if and only if it satisfies a generalization of the central limit theorem. Namely, given i.i.d. random variables $X_1, X_2,\ldots$ $\pi$-distributed, there is a slowly varying function $\ell_{\pi}$ such that 
\begin{equation}\label{eq:DefAttractionStablebis}
    \frac{\sum_{i=1}^n X_{i} - n}{n^{1/\beta}\ell_{\pi}(n)} \xrightarrow[n \to \infty]{(d)} Y,
\end{equation}
where the limit is a spectrally positive $\beta$-stable law characterized by $\mathbb E(e^{-\lambda Y})=\exp(\lambda^{\beta}), \forall~\lambda~\geq~0$.

We will mostly use the following equivalent characterization (see \cite[Proposition~2.9]{Khanfir2025HortonStrahlerNumberGaltonWatson} for a proof\footnote{Note that there is a typo in the statement of \cite[Proposition~2.9]{Khanfir2025HortonStrahlerNumberGaltonWatson}: $\frac{\alpha-1}{\Gamma(2-\alpha)}$ should be replace by its inverse $\frac{\Gamma(2-\alpha)}{\alpha-1}$.}):
$\pi \in \mathrm{dom}(\beta)$ if and only if $\psi_{\pi}:s \mapsto G_{\pi}(1-s)-(1-s)$ is $\beta$-regularly varying at $0^+$, which means that there is a slowly varying function $L_{\pi}$ such that
\begin{equation}\label{eq:CaracStableGFbis}
    \psi_{\pi}(s) = G_{\pi}(1-s) - (1-s) = s^{\beta}L_{\pi}(1/s).
\end{equation}

The slowly varying functions $\ell_{\pi}$ and $L_{\pi}$ are related as follows (see \cite{Khanfir2025HortonStrahlerNumberGaltonWatson} again).

\begin{lem}
Assume that $\pi \in \mathrm{dom}(\beta)$, then  \(x \mapsto x^{1/\beta}\ell_{\pi}(x)\) is the asymptotic inverse of \(x \mapsto\psi_{\pi}\left( \frac{1}{x} \right)^{-1}=\frac{x^{\beta}}{L_{\pi}(x)}\) at $+\infty$.
\end{lem}

These slowly varying functions influence the asymptotic behavior of $T$ a $\pi$-Bienaymé tree where $\pi \in \mathrm{dom}(\beta)$. In particular, Slack \cite{Slack1968BranchingProcessMean} obtained the following generalization of Kolmogorov's estimate for the survival probability:

\begin{equation}\label{eq:SlackResult}
{\frac{\psi_{\pi}(\P(Z_n(T)>0))}{\P(Z_n(T)>0)} \sim \frac{1}{\beta-1} \frac{1}{n}.
}\end{equation}

As a consequence, $x \mapsto \P(Z_x(T)>0))^{-1}$ is the asymptotic inverse of $x \mapsto \bigl((\beta-1)\frac{\psi_{\pi}\left( x^{-1} \right)}{x^{-1}}\bigr)^{-1}$, and $n \mapsto \P(Z_n(T)>0))$ is a $-(\frac{1}{\beta-1})$-varying function. Combining this observation with the previous lemma enables us to find an asymptotic inverse of the function \(x \mapsto \sqrt{ \frac{x}{x^{1/\alpha}\ell_{\mu}(x)} }\) appearing in Assumption~\ref{assum:JumpTailAssumption}.

\begin{lem}\label{lem:AsymptoticInversionSnake}
Assume that $\pi \in \mathrm{dom}(\beta)$ and $T$ is a $\pi$-Bienaymé tree, then  \(x \mapsto \sqrt{ \frac{x}{x^{1/\beta}\ell_{\pi}(x)} }\) has an asymptotic inverse at $+\infty$ given by 
\begin{equation*}
    x \mapsto \psi_{\pi}\bigl( \P(Z_{\frac{x^2}{\beta-1}}(T)>0) \bigr)^{-1} \sim_{x \to +\infty} \left(\frac{1}{\beta-1}\right)^{\frac{1}{\beta-1}}\frac{x^{2}}{\P(Z_{x^2}(T)>0)}.
\end{equation*}
\end{lem}

This shows that the two conditions in Assumption~\ref{assum:JumpTailAssumption} are equivalent.

\subsection{Scaling limits of trees and snakes}\label{ssct:ScalingLimitsTreesSnakes}

We finally introduce the scaling limits of a $\pi$-Bienaymé tree $T$ conditioned to be large (in several sense), namely the so-called $\beta$-stable trees. We also present the scaling limits of the corresponding $T$-snake $W$ in our specific case of centered jumps; different results exist for non-centered snakes.

For every $\beta \in (1,2]$, $\beta$-stable trees arise as scaling of large $\pi$-Bienaymé trees when $\pi \in \mathrm{dom}(\beta)$, and as one could expect they may be defined via contour functions related to (spectrally positive) $\beta$-stable Lévy process. However the construction of this contour function, known as the continuous height excursion associated with a $\beta$-stable Lévy process, is not straightforward. We give a few more details in Appendix~\ref{sct:knownScalingLimits} but refer to \cite{DuquesneLeGall2002RandomTreesLevy} for a complete picture. In this paper, we mainly need the existence of this continuous height excursion $\cH^{exc} \in C_c(\R_+,\R_+)$, but note that it is defined on a measured space $(E,\cE,\N_{\beta})$ where $\N_{\beta}$ is an infinite measure known as the Ito's excursion measure associated with a $\beta$-stable Lévy process.  We can  define various probability measures from this (again, see Appendix~\ref{sct:knownScalingLimits} for formal definitions):
\begin{itemize}
    \item $\N_{\beta}(\cdot\, |\, \max \cH^{exc} \geq 1)$;
    \item $\N_{\beta}(\cdot\, |\, \zeta(\cH^{exc}) \geq 1)$, where $[0,\zeta(\cH^{exc})]$ is the support of  $\cH^{exc}$;
    \item $\N_{\beta}^{(1)}$, which is informally speaking $\N_{\beta}(\cdot \,|\, \zeta(\cH^{exc}) = 1)$ even though $\N_{\beta}(\zeta(\cH^{exc}) = 1)=0$. When $\cH^{exc}$ is distributed under $\N_{\beta}^{(1)}$, we denote it by $\cH^{(1)}$ and call it a \emph{normalized} continuous height excursion.
\end{itemize}

We may then define a \emph{normalized $\beta$-stable tree} as the $\R$-tree $\cT_{\cH^{(1)}}$. In the special case $\beta=2$ it actually is (a multiple of) the Brownian Continuum Random tree introduced by Aldous \cite{Aldous1991ContinuumRandomTreea} when investigating the scaling limits of some combinatorial trees. More generally, Duquesne \cite{Duquesne2003LimitTheoremContour} proved that $\cT_{\cH^{(1)}}$ is the scaling limit of the Bienaymé trees $T_n$ conditioned to have precisely $n$ edges. Note that this result is with respect to the Gromov-Hausdorff topology introduced earlier, but it also holds with respect to a stronger one, namely the Gromov-Hausdorff-Prokhorov topology. We will not detail this here, but it implies that a uniform vertex $U_n$ of $T_n$ also converges in distribution to a uniform point of $\cT_{\cH^{(1)}}$.

We also say that $\cT_{\cH^{exc}}$, where $\cH^{exc}$ is distributed according to $\N_{\beta}(\cdot \,|\, \max \cH^{exc} \geq 1)$, is a \emph{$\beta$-stable tree conditioned to have height at least $1$}. As a corollary of \cite[Proposition 2.5.2]{Duquesne2003LimitTheoremContour}, which deals with convergence of contour functions, this last $\beta$-stable tree arises as scaling limit of Bienaymé trees whose height is conditioned to be large.

\begin{theorem}\label{thm:StableTreeHeight}
Assume that $\pi \in \mathrm{dom}(\beta)$, and consider $T^{H \geq n}$ a $\pi$-Bienaymé tree conditioned on $H(T) \geq n$, i.e. $Z_n(T)>0$. The following convergence in distribution holds with respect to the Gromov-Hausdorff topology.
\begin{equation*}
    \frac{1}{n} T^{H \geq n} \xrightarrow[n \to +\infty]{(d)} \cT_{\cH^{exc}},
\end{equation*}
where $\cH^{exc}$ is distributed according to $\N_{\beta}(\cdot\,|\, \max \cH^{exc} \geq 1)$.
\end{theorem}

To study scaling limit of random snakes, we also need the existence of a second process $\cS^{exc}$ defined jointly with $\cH^{exc}$ on $(E,\cE,\N_{\beta})$. Just as $\cH^{exc}$ plays the role of the limiting contour process of large trees, $\cS^{exc}$ appears as the limiting spatial contour process of large snakes. Since the spatial displacements look like Brownian motion at large scale, it is actually not difficult to define $\cS^{exc}$ given $\cH^{exc}$ as a Gaussian process (see Appendix~\ref{sct:knownScalingLimits}). The pair $(\cH^{exc},\cS^{exc})$ is sometimes referred to as \emph{the head of the Brownian snake indexed by a $\beta$-stable tree}.

In the $\beta$-stable setting, Marzouk \cite{Marzouk2020ScalingLimitsDiscrete} obtained that the contour process and spatial contour process of a $T_n$-snake, where the centered jump distribution $\theta$ satisfies Assumption~\ref{assum:JumpTailAssumption}, have $(\cH^{(1)},\cS^{(1)})$ as scaling limits. In this paper, we need a similar result under a different condition, we give its statement here and discuss how to deduce it from adaption of existing results in Appendix~\ref{sct:knownScalingLimits}.

First we introduce $k_n$ such that

\begin{equation}\label{eq:time:change}
    \sqrt{\frac{\ell_{\pi}(k_n)}{k_n^{1-1/\beta}}} \sim \frac1n.
\end{equation}

Note that $n\mapsto k_n$ exists, is $\frac{2\beta}{\beta-1}$-varying, and $k_n \sim (\frac{1}{\beta-1})^{\frac{1}{\beta-1}} \frac{n^2}{\P(Z_{n^2}(T)>0)}$ by Lemma~\ref{lem:AsymptoticInversionSnake}.

\begin{prop}\label{prop:ConvergenceSnakesCondMax}
Assume that $\pi \in \mathrm{dom}(\beta)$ and the jump distribution $\theta$ is centered and satisfies Assumption~\ref{assum:JumpTailAssumption} with $\pi$ (resp. $\beta$) instead of $\mu$ (resp. $\alpha$). Let $W_{[n]}$ be distributed as $W$ conditioned on $\max W \geq \sigma_{\theta}n$, and denote by $C_{[n]}$ and $C_{[n]}^{sp}$ its contour and spatial contour processes seen as continuous functions. Then we have the convergence in distribution 

$$\left( \frac{1}{n^2}C_{[n]}(2k_n t), \frac{1}{n \sigma_{\theta}}C_{[n]}^{sp}(2k_n t) \right)_{t \geq 0}\xrightarrow[n\rightarrow\infty]{(d)}\left(\cH^{[\geq1]}_t,\cS^{[\geq1]}_t\right)_{t\geq0},$$
where $\left(\cH^{[\geq1]}_t,\cS^{[\geq1]}_t\right)_{t\geq0}$ is distributed as $\left(\cH^{exc}_t,\cS^{exc}_t\right)_{t\geq0}$ under $\N_{\beta}(\cdot\,|\,\sup\cS^{exc} \geq 1)$ and $k_n$ is defined by \eqref{eq:time:change}. 

In addition, this jointly holds with the convergence of the size of the underlying tree $T_{[n]}$:
\begin{equation*}
    \frac{\#T_{[n]}}{k_n} \xrightarrow[n\rightarrow\infty]{(d)} \zeta\bigl( \cH^{[\geq1]} \bigr).
\end{equation*}
\end{prop}

\section{The maximum of stable branching random walk}\label{sct:TailMaxBRW}

We let aside subordination of snakes in this section and prove Theorem~\ref{thm:TailMaximumBRW}, which deals with the more classical study of the maximum of $W$ a $(\mu,\theta)$-BRW but in our $\alpha$-stable setting where $\mu \in \mathrm{dom}(\alpha)$. To ease notation, we set 
\begin{itemize}
    \item $p(n)=\P(Z_n(T)>0)$,
    \item $q(n)=\P(\max W \geq n)$,
    \item \(\psi_{\mu}(s)=G_{\mu}(1-s)-(1-s)=s^{\alpha}L_{\mu}\left( \frac{1}{s} \right)\).
\end{itemize} 

Our approach to study the tail $q$ of $\max W$ is inspired by both Lalley \& Shao \cite{LalleyShao2015MaximalDisplacementCritical} and Le Gall \& Lin \cite{LeGallLin2015RangeTreeindexedRandom}. Indeed, one of the key ingredients of our proof is that $q$ satisfies a \emph{discrete differential equation}, namely

\begin{equation}\label{eq:DiscreteODETail}
{ \mathbb E\left[ q(n-J) \right] - q(n) = \psi_{\mu} \Bigl( \mathbb E\left[ q(n-J) \right] \Bigr),
}\end{equation}
where we recall that $J$ has the distribution $\theta$. Here the \emph{differential} part comes from the left-hand side, as it
may be interpreted as \(\sigma_{\theta}^{2}/2 \times q''(n)\) for large
\(n\), while the right-hand side is asymptotic to \(\psi_{\mu}(q(n))\), which is roughly \(q(n)^{\alpha}\). 

However, instead of relating this to branching Brownian motion to get a real ODE as Lalley \& Shao and Zhang did, we use the fact that branching random walks have scaling limits (namely, the Brownian snake driven by a stable tree as shown by Marzouk \cite{Marzouk2020ScalingLimitsDiscrete}) to obtain that $q$ is regularly varying. Then, properties of regularly varying functions enable us to integrate \eqref{eq:DiscreteODETail} despite its discreteness. Let us add that the maximum of the Brownian snake driven by a stable tree also satisfies a real ODE similar to \eqref{eq:DiscreteODETail}, see Archer, Carrance \& Ménard \cite{ArcherCarranceMenard2024PropertiesStableSnakes}, but we do not use this either. Our discrete manipulations give an asymptotic characterization of $q$ that we relate to the asymptotic characterization \eqref{eq:SlackResult} of $p$ given by Slack.

To make this formal, we will need a technical lemma which gives a sufficient condition to ensure that taking an expectation preserves the regularity of a given function.

\begin{lem}\label{lem:RegularExpectation}

Consider \(Y\) a real-valued random variable, $f:\mathbb{R} \to \mathbb{R}_+$ a regularly varying function (at $+\infty$) bounded above by $1$ and \(F:\mathbb{R} \to \mathbb{R}\) a measurable function such that
\(\mathbb E\bigl[\lvert F(Y) \rvert\bigr] < +\infty\).  

If
\(\mathbb E\bigl[\lvert F(Y)\rvert\1_{\lvert Y \rvert >x}\bigr]= o_{x \to +\infty} \left( f(x) \right)\),
then
\begin{equation*}{\frac{\mathbb E\bigl[ F(Y)f(x+Y) \bigr]}{f(x)} \xrightarrow[x \to +\infty]{}  \mathbb E\bigl[F(Y)\bigr].
}\end{equation*}

\end{lem}

\begin{proof} 
Note that since \(f\) is
regularly varying we have almost surely \(\frac{f(x+Y)}{f(x)} \to 1\) as
\(x \to +\infty\). To determine the asymptotic of the expectation, we
decompose the ratio in two terms:
\begin{equation*}
{ \frac{\mathbb E\bigl[ F(Y)f(x+Y) \bigr]}{f(x)} = \mathbb E\left[ F(Y) \frac{f(x+Y)}{f(x)}\1_{\lvert Y \rvert \leq \frac{x}{2}}  \right] + \mathbb E\left[ F(Y) \frac{f(x+Y)}{f(x)}\1_{\lvert Y \rvert > \frac{x}{2}}\right].
}\end{equation*} 
By Potter's bound (Proposition~\ref{Regular:Potter}),
\(\frac{f(x+Y)}{f(x)}\1_{\lvert Y \rvert \leq \frac{x}{2}}\) is bounded
for \(x\) large enough so we can apply the dominated convergence theorem
to the first term to see that it converges to
\(\mathbb E\bigl[F(Y)\bigr]\).

For the second term, we simply write \(0 \leq f(x+Y) \leq 1\)
and then the assumption
\(\mathbb E\bigl[\lvert F(Y)\rvert \1_{\lvert Y \rvert >x}\bigr]= o \left( f(x) \right) = o \left( f(2x) \right)\)
is sufficient to prove that this term goes to \(0\).
\end{proof}

Finally, to apply the previous lemma to $q$, we will need a first lowerbound derived from a simple probabilistic observation, which relates the height of the tree to the maximal value of the snake.

\begin{lem}\label{lem:SimpleLowerboundByHeight}

There is \(c>0\) such that \(q(n) \geq c p(n^{2})\).

\end{lem}

\begin{proof} 
On the event $\{Z_{n^2}(T) > 0\}$ where $T$ has height at least $n^2$, we may bound $\max W$ from below by the maximum of the walk along a path of size $n^2$. Along this path, $W$ is the same as a classical random walk $(W'_k)_{0 \leq k \leq n^2}$, and by assumption the jumps (distributed as $J$) have a finite variance $\sigma_{\theta}^2$. Hence, we may apply Donsker's theorem to get

\begin{equation*}{
\begin{aligned}
\mathbb P(\max W \geq n) &\geq \mathbb P(\max W \geq n \ |\ Z_{n^2}(T) > 0) \mathbb P(Z_{n^2}(T) > 0) \\
& \geq \mathbb P(\max_{0 \leq k \leq n^2} W'_{k} \geq n)\mathbb P(Z_{n^2}(T) > 0) \\
& \geq \left( \mathbb P(\max_{[0,1]} \sigma_{\theta}B \geq 1) +o(1) \right)\mathbb P(Z_{n^2}(T) > 0),
\end{aligned}
}\end{equation*} 
where $B$ is a standard Brownian motion.
\end{proof}

We now prove the main result of this section.

\begin{proof}[Proof of Theorem~\ref{thm:TailMaximumBRW}]

Let us first establish the discrete differential equation \eqref{eq:DiscreteODETail}. For convenience, note that \(q:x \mapsto \P(\max W \geq x)\) is constant equal to \(1\) on
\(\mathbb{R}_{-}\), so for all \(x\) we have \(q(x_{+})=q(x)\).

For \(n \geq 1\), the branching property of \(T\) enables to express
\(1-q(n) = \mathbb P(\max W < n)\) as follows:
\begin{equation*}
{ 1-q(n) = \mathbb E\! \left[\prod_{i=1}^{Z_{1}} \bigl(1-q(n-J_{i})\bigr) \right] 
= G_{\mu}\Bigl( 1- \mathbb E\left[ q(n-J) \right]\Bigr).
}\end{equation*} 

The desired equation \eqref{eq:DiscreteODETail}, which we recall below, directly follows:

\begin{equation*}
{ \mathbb E\left[ q(n-J) \right] - q(n) = \psi_{\mu} \Bigl( \mathbb E\left[ q(n-J) \right] \Bigr).
}\end{equation*}

We then discuss the regularity of $q$. This function clearly is non-increasing and goes to $0$ as $x$ goes to $+\infty$, but we also argue that $q$ is regularly varying for some index
\(\gamma \leq 0\). Indeed, for all $\lambda > 1$ we have
\begin{equation}\label{eq:RatioTail}
{\frac{q(\lambda x)}{q(x)} = \mathbb P(\max W \geq \lambda x | \max W \geq x),
}\end{equation}
hence this ratio is the cumulative distribution function of \(\text{Law}\left( \frac{1}{x} \max W | \max W \geq x \right)\). Thanks to Proposition~\ref{prop:ConvergenceSnakesCondMax}, which applies with the change of variable $x=\sigma_{\theta}y$ since Assumption~\ref{assum:JumpTailAssumption} holds, we actually know that this distribution converges (toward $\max \cS^{[\geq1]}$, but this does not matter). Thus the ratio in \eqref{eq:RatioTail} converges for almost all $\lambda>1$ with respect to Lebesgue measure. In particular  $q$ satisfies the characterization theorem for regularly varying function (Theorem~\ref{Regular:Charac}) hence $q$ has regular variation and its index $\gamma$ must be non-positive since $q$ converges to $0$.

Thanks to this regularity of $q$, we can apply Lemma \ref{lem:RegularExpectation} with $F=1,\,f=q$ to get that
\(\mathbb E\left[ q(n-J) \right] \sim q(n)\). Note that the sufficient condition required to apply this lemma is satisfied since according to Assumption~\ref{assum:JumpTailAssumption} and Lemma~\ref{lem:SimpleLowerboundByHeight}, we have
\(\mathbb P(\lvert J \rvert \geq n) = o\left( p(n^{2})/n^{2} \right) =o(q(n))\).

Injecting \(\mathbb E\left[ q(n-J) \right] \sim q(n)\) into \eqref{eq:DiscreteODETail}
gives that \(n \mapsto \mathbb E\left[ q(n-J) \right] - q(n)\) is
\(\alpha\gamma\)-varying. To relate this function to a second derivative
of \(q\), we integrate it twice.

First, since $q(n) \to 0$, we observe that
\begin{equation}\label{eq:9960ae}
{\sum_{k \geq n} q(k-J) - q(k) = \sum_{i=1}^{J_{+}} q(n-i) -\sum_{i=0}^{J_{-}-1} q(n+i).
}\end{equation}

Since \(q\) is non-increasing and bounded by \(1\), we also have
\begin{equation*}
{\sum_{k \geq n} \lvert q(k-J) - q(k) \rvert  = \sum_{i=1}^{J_{+}} q(n-i) +\sum_{i=0}^{J_{-}-1} q(n+i) \leq \lvert J \rvert.
}\end{equation*} 
Hence,
\(h(n):= \mathbb E \left[ \sum_{k \geq n} q(k-J) - q(k) \right]\) is
well-defined and by dominated convergence we have
\begin{equation*}
{\sum_{k \geq n} \mathbb E \left[q(k-J) - q(k)\right] = h(n) <+\infty.
}\end{equation*} 

As said above,
\(n \mapsto \mathbb E\left[ q(n-J) \right] - q(n)\) is
\(\alpha\gamma\)-varying, so we must have \(\alpha\gamma \leq -1\). 
Consequently, by Karamata's Theorem~\ref{Regular:Karamata}, we get that \(h\) is \((\alpha\gamma+1)\)-varying.

Now we set
\(H(n) = \left( \sum_{k \geq n} q(k-J) - q(k) \right) -Jq(n)\). Since
\(\mathbb E[J]=0\), we have \(h(n)=\mathbb E\left[ H(n) \right]\) and
from \eqref{eq:9960ae} we also have
\begin{equation}\label{eq:0dd4bc}
{ H(n) = \sum_{i=1}^{J_{+}} \bigl( q(n-i) -q(n) \bigr)  + \sum_{i=1}^{J_{-}-1} \bigl( q(n)-q(n+i) \bigr)  \geq 0.
}\end{equation}

We deduce from \eqref{eq:0dd4bc} and \eqref{eq:9960ae} again that
\begin{equation}\label{eq:8417c9}
{ \sum_{k \geq n} H(k) = \sum_{i=1}^{J_{+}} \sum_{j=1}^{i} q(n-j) + \sum_{i=1}^{J_{-}-1}\sum_{j=0}^{i-1}q(n+j) \leq \frac{\lvert J \rvert (\lvert J \rvert +1)}{2},
}\end{equation}
hence
\(\sum_{k \geq n} h(k) = \mathbb E \left[  \sum_{k \geq n}H(k) \right] < +\infty\)
and we can (again) conclude that \(\alpha\gamma+1 \leq -1\) and that
\(n \mapsto \sum_{k \geq n}h(k)\) is \((\alpha\gamma+2)\)-varying by
Karamata's Theorem. Moreover, since \(q\) is non-increasing we can
weaken \eqref{eq:8417c9} into
\begin{equation}\label{eq:f14126}
{ \frac{J_{+}(J_{+}+1)}{2}q(n)+\frac{J_{-}(J_{-}-1)}{2} q(n+J_{-}) \leq \sum_{k\geq n} H(k) \leq \frac{J_{+}(J_{+}+1)}{2}q(n-J_{+}) + \frac{J_{-}(J_{-}-1)}{2} q(n).
}\end{equation}

Then, observe that by Assumption~\ref{assum:JumpTailAssumption} and Lemma~\ref{lem:SimpleLowerboundByHeight} we have
\(n^2\mathbb P(\lvert J \rvert\geq n) = o(p(n^2))=o(q(n))\). 
With Karamata's Theorem in hand, one readily deduces that
\begin{equation*}
    \mathbb E[J^2\1_{\lvert J \rvert \geq n}] = n^2 \mathbb P(\lvert J \rvert \geq n) + O\left(\sum_{k \geq n} k \mathbb P( \lvert J \rvert \geq k)\right) = o(q(n)) \text{ as } n \to +\infty.
\end{equation*}

Consequently, we can apply Lemma~\ref{lem:RegularExpectation} on both sides of \eqref{eq:f14126} to get
\begin{equation}\label{eq:67b0cf}
{ \frac{1}{q(n)}\sum_{k\geq n}h(k) \xrightarrow[n \to +\infty]{} \mathbb E\left[\frac{J_{+}(J_{+}+1)}{2}+\frac{J_{-}(J_{-}-1)}{2}\right]  = \frac{\sigma_{\theta}^{2}}{2}.
}\end{equation}

Thus, the \(\gamma\)-varying function \(q\) is also
\((\alpha\gamma+2)\)-varying, which means that
\(\gamma=-\frac{2}{\alpha-1}\).

Finally, we can inject \(\gamma=- \frac{2}{\alpha-1}\) into our previous
argument to get a finer asymptotic for \(q\). Indeed, since both
\(\alpha\gamma+1 = - \frac{\alpha+1}{\alpha-1} < 0\) and
\(\alpha\gamma+2 = -\frac{2}{\alpha-1} <0\), Karamata's theorem gives
the following estimates for the two integrations we have done earlier:
\begin{equation*}
{ \sum_{k\geq n}h(k) \sim \frac{-1}{\alpha\gamma+2}nh(n) \sim \frac{1}{(\alpha\gamma+2)(\alpha\gamma+1)} n^{2} \psi_{\mu}(q(n)).
}\end{equation*} 

Combining this with \eqref{eq:67b0cf} yields
\begin{equation*}
{ \frac{\psi_{\mu}(q(n))}{q(n)} \sim \frac{\alpha+1}{(\alpha-1)^2} \left( \frac{\sigma_{\theta}}{n} \right)^2.
}
\end{equation*} 

Now, recall from Slack's result \eqref{eq:SlackResult} that
\(x \mapsto \frac{1}{p(x)}\) is the asymptotic inverse of
\(x \mapsto \left( (\alpha-1) \frac{\psi_{\mu}(x^{-1})}{x^{-1}} \right)^{-1}\),
so we finally get
\begin{equation*}
{ q(\sigma_{\theta}n) \sim p(\frac{\alpha-1}{\alpha+1}n^{2}) \sim \left( \frac{\alpha+1}{\alpha-1} \right)^{\frac{1}{\alpha-1}} p(n^{2}),
}\end{equation*}
which is precisely Theorem~\ref{thm:TailMaximumBRW}.
\end{proof}

\begin{rmk}
This proof relies on the \emph{existence} alone of a scaling limit for $\max W$ condition on $\max W \geq n$, but it is possible to use existing results on this scaling limit to shorten it. Indeed, it was proved by Archer, Carrance \& Ménard \cite[Theorem 1]{ArcherCarranceMenard2024PropertiesStableSnakes} that for all $x>0$,
\begin{equation}\label{eq:ArcherCarranceMenard}
    \N_{\alpha}(\sup\cS^{exc} \geq x)=\left(\frac{\alpha+1}{(\alpha-1)^2}\frac{1}{x^2}\right)^{\frac{1}{\alpha-1}},
\end{equation}
where $\cS^{exc}$ and $\N_{\alpha}$ are defined in Appendix~\ref{sct:knownScalingLimits}. Knowing this, one can directly deduce that the ratio in \eqref{eq:RatioTail} converges to $\lambda^{-\frac{2}{\alpha-1}}$ as $x \to +\infty$, which immediately gives that $q$ has index $-\frac{2}{\alpha-1}$. However, we still need to make the same computations to obtain the exact asymptotic of $q$ and relate it to $\P(Z_{n^2}(T)>0)$, thus we chose to stick with a discrete approach. We only mention \eqref{eq:ArcherCarranceMenard} to stress that Theorem~\ref{thm:TailMaximumBRW} has a continuous counterpart. Indeed, according to \cite[Section~3]{LeGall2017SubordinationTreesBrownian}, we have 
\begin{equation*}
    \N_{\alpha}(\sup\cH^{exc} \geq x)=\left(\frac{1}{\alpha-1}\frac{1}{x}\right)^{\frac{1}{\alpha-1}},
\end{equation*}
where $\cH^{exc}$ is also defined in Appendix~\ref{sct:knownScalingLimits}. Consequently, \eqref{eq:ArcherCarranceMenard} may be written as
\begin{equation*}
    \N_{\alpha}(\sup\cS^{exc} \geq x)=\left(\frac{\alpha+1}{\alpha-1}\right)^{\frac{1}{\alpha-1}}\N_{\alpha}(\sup\cH^{exc} \geq x^2).
\end{equation*}
\end{rmk}

\section{Subordination and records of branching random walk}\label{sct:SubordinationAndRecords}

This section is devoted to the study of the subordinate tree $S(W)$ of a $T$-snake $W$ as defined in Section~\ref{ssct:DiscreteSubordination}. Recall that here $T$ is a critical or subcritical Bienaymé tree and we use centered jumps to define the $T$-walk $W$. We will mainly rely on the facts that $S(W)$ captures the genealogy of the records of $W$ (and also Theorem~\ref{thm:TailMaximumBRW}) to prove the results presented in Section~\ref{ssct:SubordinationBienaymeTrees}. We will finally investigate the relation between $S$ and the \emph{weak} records of $W$ to obtain Corollary~\ref{cor:WeakRecords}. It will be convenient to write $S(T)$ instead of $S(W)$ in this section.

\subsection{Characterization of the subordinate tree}

We first discuss the characterization given by Theorem~\ref{thm:FunctionnalCharacterization} that holds when $T$ is critical or subcritical, and then move to the more specific result Theorem~\ref{thm:NewOffspringStable} where $T$ must be a stable critical Bienaymé tree. Finally, we will end by explaining the explicit Example~\ref{ex:ExampleAndrea}.

\paragraph{General characterization (Theorem~\ref{thm:FunctionnalCharacterization}).} We begin with Proposition~\ref{prop:BienaymePreserved}, which states with almost no assumption that $S$ is also a Bienaymé tree.

\begin{proof}[Proof of Proposition~\ref{prop:BienaymePreserved}] 
Consider the $Z_1(S)$ vertices $u_1,\ldots u_{Z_1(S)}$ in $T$ where $\overline{W}$ becomes strictly positive, i.e. the first records of $W$. By construction of $S$ we have that for all $1 \leq i \leq Z_1(S),\, S^{(i)}=S(T^{(u_i)})$ where we implicitly use the restriction of $W$ on a subtree of $T$ to build its corresponding subordinate tree. But, given those $Z_1(S)$ vertices, we do not have any information on what happens after those vertices. Hence, the corresponding fringe trees $T^{(u_1)},\ldots,T^{(u_{Z_1(S)})}$ are i.i.d. copies of $T$ and the restriction of $W$ to those subtrees can be written as $W(u_1)+W^{(1)},\ldots,W(u_{Z_1(S)})+W^{(Z_1(S))}$  where the $W^{(i)}$ are i.i.d. copies of $W$. Moreover, adding a constant to a tree-indexed walk does not affect the resulting subordinate tree, so $S^{(1)},\ldots,S^{(Z_1(S))}$ are i.i.d. copies of $S$. Finally, this does not depend on the particular choice of those $Z_1(S)$ vertices, so we have the branching property.
\end{proof}

As explained in Section~\ref{ssct:SubordinationBienaymeTrees}, in the following $\nu =\text{Law}(Z_1(S))$ denotes the offspring distribution of $S(T)$. Our goal is to study the properties of $\nu$ depending on the initial offspring distribution $\mu$, but we will only answer this in the case where $\theta$ satisfies Assumption~\ref{assum:JumpAtMost1}. This assumption gives simpler relation between $S$ and $W$, namely \eqref{eq:generationRecords} and \eqref{eq:lien:hauteur}, and together with the branching property it leads to Theorem~\ref{thm:FunctionnalCharacterization} which characterizes $\nu$ and its expectation via a functional equation satisfied by $G_{\nu}$.

\begin{proof}[Proof of Theorem~\ref{thm:FunctionnalCharacterization}]

Consider the jumps $(J_i)_{1 \leq i \leq Z_1(T)}$ corresponding to the edges from the root $\varnothing$ to its children. For every child vertex $i$, either $J_i=1$ in which case $i$ is a record and belongs to the first generation of $S$, or $J_i=-k$ for some $k \geq 0$ and then $i$ is not a record but some of its descendants $v \in T^{(i)}$ may be. In this latter case, because of \eqref{eq:generationRecords}, those records are the first vertices labelled $1$ in $T^{(i)}$ and thus are precisely records of the generation $k+1$ for the branching random walk induced by $T^{(i)}$. This leads to the following decomposition of $Z_1(S)$:
\begin{equation}\label{eq:DecompositionZ1S}
    Z_1(S) = \sum_{i=1}^{Z_1(T)} \left(\mathbb 1_{(J_i=1)} + \sum_{k \geq 0}\mathbb 1_{(J_i=-k)}Z_{k+1}(S(T^{(i)})) \right)= \sum_{i=1}^{Z_1(T)} Z_{1-J_i}(S(T^{(i)})).
\end{equation}

Note that the variables $1-J_i$ are independent and distributed as $Y:=1-J$, and are also independent of the trees $S(T^{(i)})$, hence $Z_{1-J_i}(S(T^{(i)}))$ is distributed as $Z_{Y}(S)$ where $Y$ is independent from $S$. By the branching property of $S$ (Proposition~\ref{prop:BienaymePreserved}) its generating function is $\sum_{k \geq 0}\P(Y=k)G_{\nu}^{\circ k}(s)=\E[G_{\nu}^{\circ Y}(s)]$ for all $s \in [0,1]$. Then, combining \eqref{eq:DecompositionZ1S} with the branching property of $T$ yields
\begin{equation*}
G_{\nu}(s):=\E\left[s^{Z_1(S)}\right] =\E\left[\prod_{i=1}^{Z_1(T)} s^{Z_{1-J_i}(S(T^{(i)}))}\right] =G_{\mu}\bigl(\E\left[s^{Z_{Y}(S)}\right]\bigr) =G_{\mu}\bigl(\E[G_{\nu}^{\circ Y}(s)]\bigr).
\end{equation*}
This establishes the functional equation \eqref{eq:FunctionnalCharacterization}. Before discussing the uniqueness of its solution, we stress that the characterization of the mean of $\nu$, namely $\E[Z_1(S)]$, may be derived from this functional equation, but the easiest way is to take the expectation of \eqref{eq:DecompositionZ1S} and apply the branching properties again to get

\begin{equation*}
    \E[Z_1(S)] = \E[Z_1(T)]\times \E[Z_Y(S)] = m\times\E\left[\E[Z_1(S)]^Y\right].
\end{equation*}

Since $\E[Y]=1$, elementary analysis shows that $f_Y:x \in (0,1] \mapsto \frac{G_Y(x)}{x} \in [1,+\infty)$ is a decreasing bijection and $\E[Z_1(S)]=f_Y^{-1}(\frac1m)$. Moreover, it is clear that $f_Y(1)=1$, hence $\E[Z_1(S)] < 1$ if $m < 1$, else $\E[Z_1(S)]=m=1$.

Finally, to prove that $G_{\nu}$ is the unique generating function solution of \eqref{eq:FunctionnalCharacterization}, we interpret this functional equation as a fixed point equation for $\Phi:E \mapsto E$ defined by $\Phi(f)(s) = G_{\mu}\bigl(\E[f^{\circ Y}(s)]\bigr)$ on the space of functions $E = \left\{ f:[0,1] \mapsto [0,1]\ 1\textrm{-Lipschitz  and such that } f(1)=1 \right\}$. Indeed, recalling that $G_{\mu}$ is $1$-Lipschitz because $\mu$ is critical or subcritical, we see that $\Phi$ is well defined, and since any generating function solution of \eqref{eq:FunctionnalCharacterization} must be $1$-Lipschitz for the same reason, such a solution must be in $E$. We consider the uniform norm $\lVert f \rVert_{\infty} = \max_{[0,1]} \lvert f\rvert$ on $E$ (which makes it a complete space), and we prove that
\begin{equation}\label{eq:contraction}
    \forall f\neq g \in E, \lVert \Phi(f)-\Phi(g)\rVert_{\infty} < \lVert f-g\rVert_{\infty},
\end{equation}
to guarantee the uniqueness of the fixed point of $\Phi$ on $E$.

First, set $\tilde f(x) = \E[f^{\circ Y}(x)]$ and $\tilde g(x) = \E[g^{\circ Y}(x)]$, which are also $1$-Lipschitz.
For a deterministic $k \in\N$, we have 
$$ \vert f^{\circ (k+1)}(x) - g^{\circ (k+1)}(x)\rvert \leq \vert f(f^{\circ k}(x)) - f(g^{\circ k}(x))\rvert + \lvert f(g^{\circ k}(x)) - g(g^{\circ k}(x))\rvert \leq 1\times \lVert f^{\circ k} - g^{\circ k}\rVert_{\infty} + \lVert f-g\rVert_{\infty}.$$

Consequently, $\lVert f^{\circ k} - g^{\circ k}\rVert_{\infty} \leq k\times \lVert f - g\rVert_{\infty}$, and thus we get
\begin{equation*}
    \lVert \tilde f - \tilde g \rVert_{\infty} \leq \E\left[\lVert f^{\circ Y} - g^{\circ Y}\rVert_{\infty}\right] \leq \E[Y]\lVert f - g\rVert_{\infty} \leq \lVert f - g\rVert_{\infty}.
\end{equation*}

It remains to deal with the composition of $G_{\mu}$, since $\Phi(f)=G_{\mu}\circ \tilde f$. If $\tilde f(x)= \tilde g(x)$ then $ \Phi(f)(x)-\Phi(g)(x) = 0$, else $0 < G_{\mu}'(s) < 1$ on the non-trivial open interval between $\tilde f(x)$ and $\tilde g(x)$ and by the mean value theorem we have $\lvert \Phi(f)(x)-\Phi(g)(x) \rvert < \lvert \tilde f(x)-\tilde g(x)\rvert$. We may assume that $\Phi(f) \neq \Phi(g)$ (else \eqref{eq:contraction} trivially holds), in which case the continuous function $\lvert \Phi(f)-\phi(g)\rvert$ reaches its maximum $\lVert \phi(f)-\Phi(g)\rVert_{\infty} > 0$ at some point $x_* \in [0,1]$ , thus according to the previous considerations we must have $\tilde f(x_*) \neq \tilde g(x_*)$ and 
$$\lVert \Phi(f)-\Phi(g)\rVert_{\infty} = \lvert \Phi(f)(x_*)-\Phi(g)(x_*) \rvert < \lvert \tilde f(x_*)-\tilde g(x_*)\rvert \leq \lVert \tilde f - \tilde g \rVert_{\infty}.$$

The last two displays directly gives \eqref{eq:contraction} and the unicity follows.
\end{proof}

\paragraph{The case of critical $\alpha$-stable Bienaymé trees (Theorem~\ref{thm:NewOffspringStable}).}

From now on, we focus on the case where $\mu$ is critical and in the domain of attraction of an $\alpha$-stable distribution, where $\alpha \in (1,2]$. We still require that $\theta$ satisfies Assumption~\ref{assum:JumpAtMost1}, hence we already know that $\nu$ is critical, but in sight of Le Gall's result \cite{LeGall2017SubordinationTreesBrownian}, it is natural to ask whether $\nu$ is also attracted to a stable distribution. Theorem~\ref{thm:NewOffspringStable} answers this positively by establishing a Taylor expansion of $G_{\nu}$ beyond order $1$. 

Our strategy to obtain this expansion is to study $G_{\nu}$ by means of the classical relation for branching processes 
\begin{equation}\label{eq:RecExtinctionBranching}
    \P(Z_{n+1}(S)=0)=G_{\nu}\bigl(\P(Z_{n}(S)=0)\bigr).
\end{equation}

Indeed, we may translate \eqref{eq:RecExtinctionBranching} into a \emph{discrete differential equation} satisfied by $\P(Z_{n}(S)>0)$ and involving $\psi_{\nu}(s):=G_{\nu}(1-s) - (1-s)$. This will be sufficient to obtain an explicit equivalent of $\psi_{\nu}$, since under Assumption~\ref{assum:JumpAtMost1} we have a precise asymptotic of $\P(Z_{n}(S)>0)=\P(\max W \geq n)$ thanks to Theorem~\ref{thm:TailMaximumBRW}.

\begin{proof}[Proof of Theorem~\ref{thm:NewOffspringStable}]
Recall that $\psi_{\nu}(s):=G_{\nu}(1-s) - (1-s)$ is increasing (by criticality of $\nu$) and such that $\psi_{\nu}(0)=0$. Moreover, it is also clear that $q(n):= \P(\max W\geq n) = \mathbb P(Z_n(S)>0)$ decreases toward $0$, and according to \eqref{eq:RecExtinctionBranching} it satisfies the discrete differential equation

$$q(n)-q(n+1) = \psi_{\nu}(q(n)).$$

By Theorem~\ref{thm:TailMaximumBRW} we already know that $q$ is regularly varying with index $-\frac2{\alpha-1}$. Moreover $(q(n)-q(n+1))_n$ is decreasing, thus we can derive $q$ by Monotone Density Theorem  (Theorem~\ref{Regular:MonotoneDensity}) and get 
\begin{equation}\label{eq:asymp}
    \psi_{\nu}(q(n))  \sim \frac2{\alpha-1}\frac{q(n)}n.
\end{equation}

By asymptotic inversion of $q^{-1}$ (see Proposition~\ref{Regular:Inversion}), let $g$ be a regularly varying function \emph{at $0$} of index $-\frac{\alpha-1}2$ such that $n \sim g(q(n))$. We set $$h(x)=\frac2{\alpha-1}\frac{x}{g(x)},$$
which is regularly varying \emph{at $0$} of index $\frac{\alpha+1}2$. We may thus rewrite \eqref{eq:asymp} as

\begin{equation}\label{eq:AsymptoticAlongSubsequence}
    \psi_{\nu}(q(n)) \sim h(q(n)).
\end{equation}

We can extend this into $\psi_{\nu}(s) \sim h(s)$ as $s \to 0$. First, for every $s\in(0,1)$, define $N(s)$ as the unique integer such that $q(N(s)+1) < s \leq q(N(s))$. Observe that $N(s) \to +\infty$ and $q(N(s))\sim q(N(s)+1) \sim s$ as $s\to 0$ because of the monotony and regular variations of $q$. By monotony again, we have 
$$\psi_{\nu}(q(N(s)+1)) < \psi_{\nu}(s) \leq \psi_{\nu}(q(N(s))),$$
but according to \eqref{eq:AsymptoticAlongSubsequence} and the regular variations of $h$, both bounds are equivalent to $h(s)$ as $s\to 0$.

Finally, we can give an explicit expression of $g$ and thus $h$. First recall that $\psi_{\mu}(s)=G_{\mu}(1-s)-(1-s)$. We know from Theorem~\ref{thm:TailMaximumBRW} that $q(n) \sim p\left(\frac{\alpha-1}{\alpha+1}(\frac{n}{\sigma_{\theta}})^{2}\right)$,  and according to Slack's result \eqref{eq:SlackResult},
 \(x \mapsto \left( (\alpha-1) \frac{\psi_{\mu}(x^{-1})}{x^{-1}}\right)\) is the asymptotic inverse of \(x \mapsto \frac{1}{p(x)}\) at $+\infty$. This yields
 
\begin{equation*}
    g(s) \sim_{s \to 0} \sigma_{\theta}\frac{\sqrt{\alpha+1}}{\alpha-1} \left(\frac{\psi_{\mu}(s)}{s} \right)^{-1/2},
\end{equation*}
and 
\begin{equation*}
    h(s)=\frac2{\alpha-1}\frac{s}{g(s)} \sim_{s \to 0} \frac{1}{\sigma_{\theta}} \frac{2}{\sqrt{\alpha+1}} \sqrt{s\psi_{\mu}(s)},
\end{equation*}
which is the desired result since $\psi_{\mu}(s) \sim_{s \to 0} s^{\alpha}L_{\mu}(\frac1s)$.
\end{proof}

\paragraph{An explicit example (Example~\ref{ex:ExampleAndrea}).} The functional equation \eqref{eq:FunctionnalCharacterization} given in Theorem~\ref{thm:FunctionnalCharacterization} characterizes the offspring distribution $\nu$ of the subordinate tree, but given $\mu$ and $\theta$ it seems quite difficult to solve this equation. However, Theorem~\ref{thm:NewOffspringStable} enable us to guess which distribution may be realised as an offspring distribution $\nu$ of $S$, and then to look for $\mu$ and $\theta$ such as \eqref{eq:FunctionnalCharacterization} holds. To simplify this, we fix $\theta$ as the simple symmetric walk in this subsection, i.e. $\theta(1)=\theta(-1)=1/2$, hence we are looking for an offspring distribution $\mu$ such that
\begin{equation}\label{eq:FunctionalCharacInverse}
    \forall u \in [\nu(0),1], G_{\mu}^{-1} (u) = \E \left[ G_{\nu}^{\circ (Y-1)} (u) \right]=\frac{G_{\nu}^{-1} (u)+G_{\nu} (u)}{2}.
\end{equation}
Note that the right-hand side is always invertible on $[\nu(0),1]=[G_{\nu}(0),1]$, but its reciprocal map may or may not be a generating function. Hence we can state the uniqueness of the solution $\mu$ of \eqref{eq:FunctionalCharacInverse} but its existence is not guaranteed. To obtain an example, we chose a case where we can explicitly compute $G_{\mu}$ and checked that it is indeed a generating function. Most computations are actually tractable when we take  $\nu$ as a critical \emph{Power Fractional Laws}. We refer to \cite{AlsmeyerHoang2025PowerfractionalDistributionsBranching} for a survey of those distribution, which are characterized by their generating function
\begin{equation*}
    G_{\nu}(x) = G_{b,\kappa}(x):= 1-\left( \frac{1}{(1-x)^{\kappa}}+b \right)^{-1/\kappa} = 1- \frac{1-x}{(1+b(1-x)^{\kappa})^{1/\kappa}},
\end{equation*}
where $b >0$ and $\kappa \in (0,1]$. Indeed, those generating functions are well behaved with respect to composition as one can check that $G_{b,\kappa}^{\circ k} = G_{kb,\kappa}$ for all $k \in \Z$ (for negative $k$, $G_{kb,\kappa}$ is still defined as above but it is not a generating function anymore). An expansion at $1$ also reveals that such an offspring distribution $\nu$ is critical and attracted to an $(1+\kappa)$-stable law, hence we can fix $\kappa \leq \frac12$ so that $\nu$ may satisfies the conclusion of Theorem~\ref{thm:NewOffspringStable}.

Example~\ref{ex:ExampleAndrea}, where $\nu(n) = \frac{3}{2} \frac{C_{n}}{4^{n}(n+2)}$, actually corresponds to a critical Power Fractional Law with $b=1,\kappa=\frac12$, since identifying $G_{\nu}$ as a hypergeometric series yields

\begin{equation}\label{eq:exAndreaNu}
    G_{\nu}(x) = \sum_{n \in\N} \frac{3}{2\times4^{n}} \frac{(2n)!}{n!(n+2)!} x^n =  \frac{3x-2+2(1-x)^{3/2}}{x^{2}} = 1-\frac{1-x}{(1+\sqrt{ 1-x })^{2}} = G_{1,\frac12}(x).
\end{equation}

\begin{proof}[Proof of Example~\ref{ex:ExampleAndrea}]
By \eqref{eq:exAndreaNu}, $G_{\nu}=G_{1,\frac12}$, hence according to the previous discussion $G_{\nu}^{-1}(x)=G_{-1,\frac12}(x)=1-\frac{1-x}{(1-\sqrt{ 1-x })^{2}}$ and we are looking for a distribution $\mu$ satisfying \eqref{eq:FunctionalCharacInverse}, which amounts to 
\begin{equation*}
    \forall u \in [\nu(0),1], G_{\mu}^{-1} (u)  = \frac{3u-2}{u^2}.
\end{equation*}

Set $x=G_{\mu}^{-1} (u)$, we get that $u \in [\nu(0),1]$ satisfies $xu^2-3u+2=0$ and solving this provide the expression
\begin{equation}\label{eq:exAndreaMu}
    \forall x \in [G_{\mu}^{-1}(\nu(0)),1],\ G_{\mu}(x) = u = \frac{3-\sqrt{ 9-8x }}{2x}.
\end{equation}

This is indeed the generating function of some distribution, since with $C(x) = \sum_{n \geq  0} C_{n} x^{n} =\frac{1-\sqrt{ 1-4x }}{2x}$ we can write $G_{\mu}(x)=\frac{2}{3}C\left( \frac{2}{9} x \right)$. We conclude that $G_{\mu}$ is the power series with coefficient $\mu(n)=\frac{2}{3} \left( \frac{2}{9} \right)^{n} C_{n}$ as claimed.
\end{proof}

\begin{rmk}
For an initial choice of $\nu$ such that $G_{\nu}=G_{b,\kappa}$ and for more general $\theta$, it is still possible to compute $G_{\mu}^{-1}$ as above, however deducing that the corresponding $G_{\mu}$ is indeed a generating function is not as straightforward as in Example~\ref{ex:ExampleAndrea} and actually out of our reach in most cases. Let us just mention that for the same $\theta$ and $\kappa=\frac12$, we can reduce the case $b>0$ to the case $b=1$ but we do not have explicit expressions for $\nu(n)$ and $\mu(n)$ anymore. Then for $b=1,\kappa=\frac12$ and $\theta$ given by $\theta(1)=\theta(-1)=p$ and $\theta(0)=1-2p$, $G_{\mu}$ is explicitly computable but quite tedious to write as we have to solve a cubic equation. Finally, for $\kappa<\frac12$, the equation appearing to invert $G_{\mu}^{-1}$ cannot be algebraically solved hence we were not able to prove the existence of a corresponding $\mu$.
\end{rmk}

\subsection{Application to the study of weak records (Corollary~\ref{cor:WeakRecords})}

In the previous section, we mainly used that $S:=S(W)$ captures the genealogy of records in $W$ to study it. The same approach is available to study the tree $\widetilde{S}:=\widetilde{S}(W)$ that captures the genealogy of \emph{weak records} in $W$. First, this will easily lead to the fact that $\widetilde{S}$ is a Bienaymé tree with some offspring distribution $\widetilde{\nu}$. Then, this approach will give direct relations between $\nu$ and $\widetilde{\nu}$ rather than fixed point equations, hence the behavior of $\widetilde{S}$ stated in Corollary~\ref{cor:WeakRecords} will be deduced from the previous results on $S$.

Formally, we say that $u \in T$ is a weak record for the $T$-walk $W$ when $\forall v \preceq u,\ W(v) \leq W(u)$. We then define $\widetilde{S}(W)$ by contracting (i.e. setting length to $0$) all edges that do not lead to a weak record. With this definition and the ideas used to study $S$, we now prove Corollary~\ref{cor:WeakRecords}. 

\begin{proof}[Proof of Corollary~\ref{cor:WeakRecords}.]
\begin{enumerate}[(i)]
    \item $\widetilde{S}$ has the branching property for the same reasons $S$ has this property. Indeed the proof of Proposition~\ref{prop:BienaymePreserved} also applies to $\widetilde{S}$ as long as we consider the $Z_1(\widetilde{S})$ vertices $\widetilde{u}_1,\ldots, \widetilde{u}_{Z_1(\widetilde{S})}$ that are the first weak records coming after the root (instead of the first records).
    \item To characterize the offspring distribution $\widetilde{\nu}$, i.e. the law of $Z_1(\widetilde{S})$, we decompose $T$ at its root according to the jumps $(J_i)_{i \leq i \leq Z_1(T)}$ from this root to its children, as in the proof of Theorem~\ref{thm:FunctionnalCharacterization}. This time, for every child vertex $i$, either $J_i \in \{0,1\}$ in which case $i$ is a weak record and belongs to the first generation of $\widetilde{S}$, or $J_i=-k$ for some $k \geq 1$ and one has to look to the descendants of $i$ to find a potential weak record. But once restricted to this subtree $T^{(i)}$, the vertices we are interested in are the first ones with labels $0$, hence they correspond to (strong) records in the $T^{(i)}$-walk. This leads to the following decomposition of $Z_1(\widetilde{S})$:
    
    \begin{equation}\label{eq:DecWeakRecords}
        Z_1(\widetilde{S}) = \sum_{i=1}^{Z_1(T)} \left( \1_{J_i \geq 0} + \sum_{k \geq 1} \1_{J_{i}=-k}Z_{k}(S(T^{(i)})) \right) = \sum_{i=1}^{Z_{1}(T)}Z_{(J_{i})_{-}}(S(T^{(i)})), 
    \end{equation}
    where we used that $Z_0(S)=1$ for any subordinate tree. 
    Just as for Theorem~\ref{thm:FunctionnalCharacterization}, by independence and the branching property of $S$ we can deduce the generating function of $Z_1(\widetilde{S})$ from this decomposition. We get 
    \begin{equation}\label{eq:CharacWeakRecords}
        \forall s \in [0,1], G_{\widetilde{\nu}}(s) = G_{\mu}\bigl( \E[G_{\nu}^{\circ J_{-}}(s)] \bigr).
    \end{equation}
    
    Taking the expectation of \eqref{eq:DecWeakRecords} also gives $\E[Z_{1}(\widetilde{S})]= m G_{J_{-}}(\E[Z_{1}(S)])$ where $m=\E[Z_{1}(T)]$ is the mean of $\mu$ and $\E[Z_{1}(S)]$ is the mean of $\nu$ (characterized by Theorem~\ref{thm:FunctionnalCharacterization}). We clearly see that when $m=1$, we have $\E[Z_{1}(S)]=1$ and thus $\E[Z_{1}(\widetilde{S})]=1$ too. On the contrary, when $m < 1$ we get $\E[Z_{1}(\widetilde{S})]<1$.
    
    \item Based on the similarity between \eqref{eq:CharacWeakRecords} and the functional equation characterizing $G_{\nu}$ (Theorem~\ref{thm:FunctionnalCharacterization}), we relate their expansion at $1$. To ease notation, set $\widetilde{\psi}(s)= G_{\widetilde{\nu}}(1-s)-(1-s)$, $\psi(s)= G_{\nu}(1-s)-(1-s)$ and $\psi_{\mu}(s)= G_{\mu}(1-s)-(1-s)$. Those functions are increasing and we know that $\psi_{\mu}$ is $\alpha$-varying at $0$ while $\psi(s)$ is $\frac{\alpha+1}{2}$-varying (Theorem~\ref{thm:NewOffspringStable}). We also set $\gamma(s)=1- \E[G_{\nu}^{\circ (1-J)}(1-s)]$ and $\widetilde{\gamma}(s)=1- \E[G_{\nu}^{\circ J_{-}}(1-s)]$, which are both $\leq s$ by convexity of critical generating functions. Then the equations characterizing $\nu$ and $\widetilde{\nu}$ lead to
    \begin{equation*}
\begin{aligned}
\psi(s) -  \widetilde{\psi}(s)&= G_{\nu}(1-s) - G_{\widetilde{\nu}}(1-s) \\
&=G_{\mu}(1-\gamma(s))-G_{\mu}(1-\widetilde{\gamma}(s)) \\
&= \widetilde{\gamma}(s)-\gamma(s) -\psi_{\mu}(\widetilde{\gamma}(s))+\psi_{\mu}(\gamma(s)) \\
&= \E \bigl[ G_{\nu}^{\circ (1-J)}(1-s) -  G_{\nu}^{\circ J_{-}}(1-s)\bigr] +O(\psi_{\mu}(s)).
\end{aligned}
    \end{equation*}
    
To deal with this expectation, first observe that for any realization of $J$ we always have
\begin{equation*}
    G_{\nu}^{\circ (1-J)}(1-s) -  G_{\nu}^{\circ J_{-}}(1-s) = \1_{J \leq 0} \psi(1-G_{\nu}^{\circ (-J)}(1-s)) \sim_{s \to 0} \1_{J \leq 0} \psi(s),
\end{equation*}
because $\psi$ is regularly varying and $1-G(1-s) \sim s$ as $s \to 0$ for any critical generating function $G$. Moreover, we also have  $1-G(1-s) \leq s$ for $s \in [0,1]$ and $\psi$ is increasing, thus the previous display also gives
\begin{equation*}
    \forall s \in [0,1], \frac{G_{\nu}^{\circ (1-J)}(1-s) -  G_{\nu}^{\circ J_{-}}(1-s) }{\psi(s)} \leq  \1_{_{J \leq 0}}.
\end{equation*}
According to those three last displays, we can apply Lebesgue's dominated convergence theorem to obtain
\begin{equation*}
    \psi(s)-\widetilde{\psi}(s) = \bigl(\P(J \leq 0) + o(1)\bigr)\psi(s) + O(\psi_{\mu}(s)).
\end{equation*}

Since $\alpha > \frac{\alpha+1}{2}$, we have $\psi_{\mu}(s) = o(\psi(s))$ hence we finally get $\widetilde{\psi}(s) \sim_{s \to 0} \theta(1)\psi(s)$. The desired result follows from this and Theorem~\ref{thm:FunctionnalCharacterization}.
\end{enumerate}
\end{proof}

\section{Scaling limits for subordinate trees}\label{sct:ScalinglimitsSubordination}

In this section, we investigate the asymptotic behavior of subordinate trees associated with large random snakes.

\subsection{Conditioning on the maximum of the snake (Corollary~\ref{cor:ScalingLimitSubordinateCondMax})}

We have seen that $S(W)$ is a function of $W$, in particular by \eqref{eq:ContourSubordinate} we know how to express a contour process of $S(W)$ by means of the path representation of $W$. We will argue that this function actually is continuous, hence one may obtain joint convergence result directly from convergence results for $W$. However, the dependence between $W$ and $S(W)$ may not be tractable in general. Hence, even thought we have a joint convergence, the scaling limit of $S(W)$ may not be explicit. 
Corollary~\ref{cor:ScalingLimitSubordinateCondMax} describes this joint convergence and the explicit limit under the specific conditioning $\max W \geq \sigma_{\theta} n$, and this section is devoted to its proof. The reason why this conditioning is convenient is that given Assumption~\ref{assum:JumpAtMost1}, we can use \eqref{eq:lien:hauteur} which says that the condition $\max W \geq k$ is strictly equivalent to $Z_k\left(S(W)\right) > 0$. Thus, in this specific case, the subordinate tree simply is a Bienaymé tree conditioned on its height and its scaling limit is given by Theorem~\ref{thm:StableTreeHeight}.

\begin{proof}[Proof of Corollary~\ref{cor:ScalingLimitSubordinateCondMax}]
Recall that $W_{[n]}$ is distributed as a $(\mu,\theta)$-random snake $W$ conditioned on $\max W \geq \sigma_{\theta}$. We begin with the convergence of  $S(W_{[n]})$ jointly with $(C_{[n]},C_{[n]}^{sp})$, which are the contour and spatial contour processes of $W_{[n]}$. First, define $\overline{C_{[n]}}$ by considering the contour exploration $u_0,u_1,\ldots$ of $T_{[n]}$ and setting
\begin{equation*}
    \forall i \in \llbracket 0, 2\#T_{[n]}\rrbracket,\ \overline{C_{[n]}}(i) = \overline{W_{[n]}}(u_i) = \max_{v \preceq u_i} W_{[n]}(v).
\end{equation*}

According to \eqref{eq:ContourSubordinate}, we know that $\overline{C_{[n]}}$ is a contour process of $S(W_{[n]})$. Then we extend $(C_{[n]},C_{[n]}^{sp},\overline{C_{[n]}})$ by linear interpolation, so that $(C_{[n]},C_{[n]}^{sp},2\#T_{[n]}) \in \T$ where $\T$ is the state space of tours defined in Section~\ref{ssct:ConvTreeSnake}. We may then consider $(C_{[n]},\omega_{[n]},2\#T_{[n]}) = \mathfrak H^{-1}\left(C_{[n]},C_{[n]}^{sp},2\#T_{[n]}\right) \in \W$ where $\mathfrak H$ is the homeomorphism from Theorem~\ref{thm:homeo}. By definition, $\omega_{[n]} \in C_c(\R_+,C(\R_+,\R))$ is a continuous extension of the representation of the snake $W_{[n]}$ by a collection of discrete paths, and one can easily check that 
\begin{equation*}
    \forall s \geq 0, \overline{C_{[n]}}(s) = \max \left(\omega_{[n]}\right)_s.
\end{equation*}

Moreover, the map $F:\W \mapsto C_c(\R_+,\R)$ defined by $F(h,\omega,\zeta)=(s\mapsto \max \omega_s)$ is $1\mathrm{-Lip}$ for $\lVert \cdot \rVert_{\infty}$ and $d_{\W}$. Hence $\overline{C_{[n]}}$ depends on $(C_{[n]},C_{[n]}^{sp},2\#T_{[n]})$ in a continuous way, and the same holds for the renormalized processes:
\begin{equation*}
    \frac{1}{\sigma_{\theta}n}\overline{C_{[n]}}(2k_n\cdot) = F\circ \mathfrak H^{-1}\left( \frac{1}{n^2}C_{[n]}(2k_n \cdot), \frac{1}{n \sigma_{\theta}}C_{[n]}^{sp}(2k_n \cdot), \frac{\#T_{[n]}}{k_n} \right),
\end{equation*}
where $k_n$ was defined by \eqref{eq:time:change}. Our point is that by Proposition~\ref{prop:ConvergenceSnakesCondMax} and continuity of both $F \circ \mathfrak H^{-1}$ and $f \mapsto \cT_f$ (recall Section~\ref{ssct:ConvTreeSnake}) we have

\begin{equation*}
    \left( \frac{1}{n^2}C_{[n]}(2k_n \cdot), \frac{1}{n \sigma_{\theta}}C_{[n]}^{sp}(2k_n \cdot), \frac{1}{\sigma_{\theta}n}S(W_{[n]}) \right) \xrightarrow[n \to +\infty]{(d)} \biggl(\cH^{[\geq1]},\cS^{[\geq1]}, \cT_{F \circ \mathfrak H^{-1}\bigl( \cH^{[\geq1]},\cS^{[\geq1]}, \zeta ( \cH^{[\geq1]} ) \bigr)} \biggr).
\end{equation*}

However, the contour function $\cC = F \circ \mathfrak H^{-1}\bigl( \cH^{[\geq1]},\cS^{[\geq1]}, \zeta ( \cH^{[\geq1]} ) \bigr)$ is not a known contour function associated with a classical scaling limit of trees. Nonetheless, in the particular case of $W_{[n]}$, according to \eqref{eq:lien:hauteur} and Theorem~\ref{thm:NewOffspringStable}, $S(W_{[n]})$ simply is a $\nu$-Bienaymé tree conditioned to have height at least $\sigma_{\theta}n$. Hence, we may apply Theorem~\ref{thm:StableTreeHeight} to deduce that $\cT_{\cC}$ actually is an $\frac{\alpha+1}{2}$-stable tree conditioned to have height at least $1$, which concludes the proof.

\end{proof}

\subsection{Conditioning on the size of the tree (Theorem~\ref{thm:expectedNumberOfRecords})}

We now turn to the more classical case of a $T_n$-indexed random walk $W_n$, where $T_n$ is a $\mu$-Bienaymé tree conditioned to have $n$ edges. In this case, $\#T_n$ is not a deterministic function of $S(W_n)$, hence we cannot construct $S(W_n)$ as a $\nu$-Bienaymé tree satisfying some intrinsic condition and we have to take into account the initial tree $T_n$. The expected size $\E\left[ \#S(W_n)\right]$ is still tractable and the point of this section is to prove Theorem~\ref{thm:expectedNumberOfRecords} which provides its asymptotic.

The main idea to obtain Theorem~\ref{thm:expectedNumberOfRecords} is to consider a random vertex $U_{n}$ in $T_n$ and look at the probability that it realises a record in $W_n$. Indeed, this is directly linked to the expected number of records $\E\left[1+ \#S(W_n)\right]$ but it may also be expressed as the probability that a classical random walk $X = (X_n)_{n \geq 0}$ with jump distribution $\theta$ has a record at position $\lvert U_n \rvert$. 

On the one hand, this height $\lvert U_n \rvert$ is very well understood. To see this, first recall Duquesne's theorem \cite{Duquesne2003LimitTheoremContour}. To ease notation, for the rest of this section, we set

\begin{equation*}
    h_n = \frac{n^{1-1/\alpha}}{\ell_{\mu}(n)}.
\end{equation*}
Duquesne's theorem gives the convergence of $\frac{1}{h_{n}}T_{n}$ toward an $\alpha$-stable tree with respect to the Gromov-Hausdorff-Prokhorov distance. In particular, this implies

\begin{equation}\label{eq:cvHauteurUnif}
    \frac{\lvert U_{n} \rvert}{h_{n}} \xrightarrow[n \to +\infty]{(d)} \mathcal{U}_{\alpha}
\end{equation}
where $\mathcal{U}_{\alpha}$ is the height of a uniformly chosen vertex of this $\alpha$-table tree. In addition, \cite[Theorem~3.3.3]{DuquesneLeGall2002RandomTreesLevy} tells us that $\alpha \times  \mathcal{U}_{\alpha}$ has an explicit density, which may be written as 

$$
x \mapsto \Gamma\left( 1+\frac{\alpha-1}{\alpha} \right) x \cdot\mathrm{ML}_{\frac{\alpha-1}{\alpha}}(x),
$$
where $\mathrm{ML}_{\beta}$ is the density of a Mittag-Leffler distribution associated with a $\beta$-subordinator. This shows that for all $p>-2$, we have

\begin{equation}\label{eq:MomentshauteurUnif}
    \E[\mathcal{U}_{\alpha}^{p}]= \alpha^{-p}\Gamma\left( 1+\frac{\alpha-1}{\alpha} \right) \E\left[ \left( \mathcal{M}_{\frac{\alpha-1}{\alpha}} \right)^{p+1} \right]= \alpha^{-p} \frac{\Gamma\left( 1+\frac{\alpha-1}{\alpha} \right)\Gamma(p+2)}{\Gamma\left( 1+(p+1)\frac{\alpha-1}{\alpha} \right)},
\end{equation}
where $\mathcal{M}_{\beta}$ follows a Mittag-Leffler distribution whose moments are given by \cite[Equation~0.42]{PitmanPicard2006CombinatorialStochasticProcesses}.

On the other hand, the probability that $X$ has a record at a given position $k$ is also well understood (see e.g. \cite[XII.7 Theorem~1a]{Feller1971IntroductionProbabilityTheory}). We recall here its asymptotic, with a simple proof for completeness.

\begin{lem}\label{lem:ClassicalRecord}
Consider $X$ a (classical) random walk whose jump distribution $\theta$ is centered, has a finite variance, and satisfies Assumption~\ref{assum:JumpAtMost1}. The asymptotic of $f_{\theta}(k)=\P(X \text{ has a record at time } k)$ is given by
\begin{equation*}
    f_{\theta}(k) \sim_{k\rightarrow\infty} \frac{\sigma_{\theta}}{\sqrt{ 2\pi }} k^{-1/2}.
\end{equation*}
\end{lem}
\begin{proof}
First, by duality we have that $\P(X \text{ has a record at time } k)=\P(X_{j}>0\  \forall j \in \llbracket 1,k \rrbracket  )$, hence $f_{\theta}$ is decreasing. Moreover, because of Assumption~\ref{assum:JumpAtMost1}, the number of records of $X$ up to time $k$ is exactly $1+\max_{0 \leq  j \leq k} X_{j}$. It is then sufficient to prove that

\begin{equation}\label{eq:ClassicalExpectedRecord}
    \sum_{j=1}^k f_{\theta}(j) = \E[1+\max_{0 \leq  j \leq k} X_{j}] \sim \sigma_{\theta}\sqrt{ \frac{2}{\pi} } k^{1/2},
\end{equation}
since by Monotone Density Theorem (Theorem~\ref{Regular:MonotoneDensity}) we can derive this to get Lemma~\ref{lem:ClassicalRecord}.

Notice that by Donsker's theorem, we already have the convergence in distribution

\begin{equation}\label{eq:DonskerMax}
   \frac{1+\max_{0 \leq  j \leq  k} X_{j}}{k^{1/2}} \xrightarrow[]{(d)} \sup_{0 \leq t \leq 1} (\sigma_{\theta}B_{t}) \stackrel{(d)}{=} \sigma_{\theta}\lvert B_{1} \rvert, 
\end{equation}
where $(B_{t})_{t\geq 0}$ is a standard Brownian motion. We can actually reinforce it into a $\L^1$ convergence. Indeed, by Doob's maximal inequality, we have

$$
\E\left[ \left( \frac{\max_{0 \leq  j \leq k} X_{j}}{k^{1/2}} \right)^{2} \right] \leq  4 \frac{\E[X_{k}^{2}]}{k} = 4 \sigma_{\theta}^{2},
$$
which implies that we have uniform integrability and the convergence \eqref{eq:DonskerMax} also holds in expectation. Since $\E[\lvert B_{1} \rvert]=\sqrt{ \frac{2}{\pi} }$, we obtain the desired Equation~\eqref{eq:ClassicalExpectedRecord}.

\end{proof}

With the behaviours of $\lvert U_n \rvert$ and $f_{\theta}$ in hand, we can now prove Theorem~\ref{thm:expectedNumberOfRecords} with the same method used above for Lemma~\ref{lem:ClassicalRecord}. Namely, we will isolate a leading term and a random term that converges in distribution, and then prove that it actually converges in expectation. To this aim, we actually need an additional lemma on $\lvert U_n \rvert$ to ensure uniform integrability.

\begin{lem}\label{lem:UnifBoundHeight}
Assume that $\mu \in \mathrm{dom}(\alpha)$. There is $C=C(\mu)$ such that for all $n \geq 1$ and $y>0$,

$$
\P(1 \leq \lvert U_{n} \rvert \leq h_{n}y ) \leq  C(\mu)y,
$$
where $h_{n}=\frac{n}{n^{1/\alpha}\ell_{\mu}(n)}$ as above.
\end{lem}

\begin{proof}
This result relies on a similar bound obtained by Kortchemski \cite[Theorem 1]{Kortchemski2017SubexponentialTailBounds} for the width of $T_n$ defined as $\mathrm{Width}(T_n)=\max_{k\in\N} Z_k(T_n)$. With our notation, Kortchemski's bound gives the existence of some constants $C_1,C_2$ such that 

$$
 \P\left( \mathrm{Width}(T_n) \geq u \times  \frac{n}{h_n}
 \right) \leq  C_{1} \exp(-C_{2}u).
$$

This is useful here since we obviously have for every $k\geq 0$, $\#\{ u \in T_n: 1 \leq \lvert u \rvert \leq k \} \leq k \times \mathrm{Width}(T_n)$, hence we may write

$$
\begin{aligned}
\P(1 \leq  \lvert U_{n} \rvert \leq  h_{n}y) &\leq  \frac{yh_{n}\times \E\left[\mathrm{Width}(T_n)\right]}{n}  \\
&\leq y\int_{0}^{+\infty} \P\left( \mathrm{Width}(T_n) \geq u \times  \frac{n}{h_n}
 \right) \dd u \\
 &\leq y \times  C_{1}\int_{0}^{+\infty} \exp(-C_{2}u)\dd  u.
\end{aligned}
$$

This last integral is finite, hence we get Lemma~\ref{lem:UnifBoundHeight}.
\end{proof}

\begin{rmk}
Lemma~\ref{lem:UnifBoundHeight} will be sufficient to compute a first moment here, but it may not be optimal since $\frac{\left\lvert  U_{n}  \right\rvert}{h_{n}}$ converges to $\mathcal{U}_{\alpha}$, whose density is recalled above and satisfies

 $$
 \P(\mathcal{U}_{\alpha} \leq y) \sim_{y \to 0} C'(\alpha)y^{2}.
 $$

\end{rmk}

\begin{proof}[Proof of Theorem~\ref{thm:expectedNumberOfRecords}]
Let $U_{n}$ be a uniform vertex of $T_{n}$. On the one hand we have

$$
\P( U_{n} \text{ is a record}) = \frac{\E\left[1+\#S(W_n)\right]}{n+1}.
$$

On the other hand, since being a record only depends on the linear random walk observed along one's ancestral line,

$$
\P(U_{n} \text{ is a record}) = \E[f_{\theta}(\lvert U_{n} \rvert )],
$$
where $f_{\theta}(k)$ is defined by Lemma~\ref{lem:ClassicalRecord}.

Since we have the convergence \eqref{eq:cvHauteurUnif} for $\lvert U_{n} \rvert$ and $f_{\theta}(k)$ has order $\frac{1}{\sqrt{ k }}$, this leads to the following expression for $\E\left[\#S(W_n)\right]$:

\begin{equation}
    \frac{1}{\sigma_{\theta}}\E\left[1+\#S(W_n)\right] = \frac{n+1}{\sqrt{ h_{n} }} \times  \E\left[ \frac{\sqrt{ h_{n} }f_{\theta}(\lvert U_{n} \rvert )}{\sigma_{\theta}} \right].
\end{equation}

Notice that $\frac{n}{\sqrt{ h_{n} }} = \ell_{\mu}(n)^{1/2}n^{\frac{\alpha+1}{2\alpha}}$ as desired, hence we must prove that $\E\left[ \frac{\sqrt{ h_{n} }f_{\theta}(\lvert U_{n} \rvert )}{\sigma_{\theta}} \right]$ converges to obtain Theorem~\ref{thm:expectedNumberOfRecords}.

By Skorokhod's representation theorem applied to \eqref{eq:cvHauteurUnif}, we may assume that $\frac{\lvert U_{n} \rvert}{h_{n}} \xrightarrow[n \to +\infty]{a.s.} \mathcal{U}_{\alpha}$ and then use Lemma~\ref{lem:ClassicalRecord} to get 

\begin{equation}\label{eq:946671}
    \frac{\sqrt{ h_{n} }f_{\theta}(\lvert U_{n} \rvert)}{\sigma_{\theta}} \xrightarrow[n \to +\infty]{(d)} \frac{1}{\sqrt{ 2\pi \mathcal{U}_{\alpha} }}.
\end{equation}

Moreover, according to \eqref{eq:MomentshauteurUnif}, the limiting distribution in \eqref{eq:946671} has a finite expectation precisely given by 

\begin{equation*}
    \E\left[ (2\pi \mathcal{U}_{\alpha})^{-1/2}\right] = \frac{1}{2}\sqrt{ \frac{\alpha}{2} }\frac{\Gamma\left( 1+\frac{\alpha-1}{\alpha} \right)}{\Gamma\left( 1+\frac{\alpha-1}{2\alpha} \right)}=\sqrt{ \frac{\alpha}{2\pi} }\frac{\Gamma\left(1-\frac{1}{2\alpha} \right)}{2^{1/\alpha}}:= c(\alpha),
\end{equation*}
where we used the duplication formula for $\Gamma$. It only remains to prove that  \eqref{eq:946671} also holds in expectation to conclude the proof. To do so, we obtain some uniform integrability by establishing a bound in $\L^{p}$ for some $p>1$.

First, by Lemma~\ref{lem:ClassicalRecord}, we have a constant $C=C(\theta)$ such that $f_{\theta}(k) \leq  \frac{C}{\sqrt{ 1 \vee k }} \forall k \geq  0$, which gives the following bound (with a finite constant $C_{p}$):

$$
\begin{aligned}
\E\left[ \left( \frac{\sqrt{ h_{n} }f_{\theta}(\lvert U_{n} \rvert )}{\sigma_{\theta}} \right)^{p} \right] &\leq C_{p} \E\left[  \left( \frac{h_{n}}{1\vee \lvert U_{n} \rvert } \right)^{\frac{p}{2}}  \right]  \\
& \leq   C_{p} \int_{0}^{+\infty} \P(1\vee \lvert U_{n} \rvert \leq  h_{n} x^{-2/p}) \dd x \\
&\leq C_p\left( \int_{0}^{+\infty} \left(\P(1 \leq   \lvert U_{n} \rvert \leq  h_{n} x^{-2/p})+\mathds{1}_{\{h_n x^{-2/p}>1\}}\P(|U_n|=0)\right) \dd x\right) \\
& \leq C_{p} \left( \int_{0}^{+\infty} \P(1 \leq   \lvert U_{n} \rvert \leq  h_{n} x^{-2/p}) \dd x + \frac{(h_{n})^{\frac{p}{2}}}{n}\right)
\end{aligned}
$$

By choosing $p \in (1,\frac{2\alpha}{\alpha-1})$, we ensure that $(h_{n})^{\frac{p}{2}} = o(n)$. To bound from above the remaining integral uniformly in $n$, we apply Lemma~\ref{lem:UnifBoundHeight} to get

$$
\E\left[ \left( \frac{\sqrt{ h_{n} }f_{\theta}(\lvert U_{n} \rvert )}{\sigma_{\theta}} \right)^{p} \right] \leq C'_{p} \left( \int_{0}^{+\infty} (1 \wedge x^{-2/p})\dd x  + 1\right).
$$

The last bound is finite when $p \in (1,2) \subset (1,\frac{2\alpha}{\alpha-1})$, hence we have a uniform bound in $\L^{p}$ for those $p$, and the result follows.
\end{proof}

\appendix

\section{More background on scaling limits of snakes}\label{sct:knownScalingLimits}

The purpose of this section is to present the scaling limit results for discrete snakes that we rely on throughout this article. In this section, $\beta\in(1,2]$ is fixed.

\paragraph{Encoding processes and excursion theory.}

Let $\cX=(\cX_t)_{t\geq0}$ be a spectrally positive stable Levy process with index $\beta$. Since it has no negative jumps, this process is characterized by its Laplace exponent; we normalize it so that for every $t\geq0$ and $\lambda>0$, 
\[\E\left[\exp(-\lambda \cX_t)\right]=\exp(-t\lambda^\beta).\]
As explained by Duquesne \& Le Gall \cite{DuquesneLeGall2002RandomTreesLevy}, where the reader may find a detailed version of what follows, one can define the \textit{height process} $\cH$ associated with $\cX$. This process $\cH$ is a continuous function that will be interpreted as a contour function for an infinite forest of real trees. One can then equip those trees with Brownian motion along every branches to get the \textit{Brownian snake driven by $\cH$}, but here it will be easier to consider the associated spatial contour process $\cS$. Indeed, $\cS$ may be directly defined as a Gaussian process given $\cH$ characterized by continuous trajectories and the following covariance function: for every $0\leq s\leq t$, we have

\[\E\left[(\cS_t-\cS_s)^2\,|\,\cH\right]=\cH_t+\cH_s-2\min_{r\in[s,t]}\cH_r.\]

Finally, we also define $\cI$ as the opposite of the running infimum of $\cX$, meaning that $\cI_t = - \inf_{s \leq t} \cX_s$. It is a continuous non-decreasing process, and $\cI_t$ may be interpreted as the amount of trees in the forest encoded by $(\cH_s)_{0 \leq s \leq t}$. In particular, the random measure $\dd\cI_s$ is exactly supported on the set $\{s \geq 0: \cH_s = 0\}$.

In order to work with a single random tree instead of a forest, we also want to consider a single excursion of $\cX$ above $0$ and the associated processes. Let us first recall the structure of the excursions in $\cX$:  we can decompose $\cX$ into a collection $(e^i)_i$ of excursions above its running infimum. More precisely, we write $(g_i,d_i)$ for an excursion interval of $\cX+\cI$ away from $0$, and we set $e^i(s)=(\cX_{g_i+s}-\cX_{g_i})\1_{s \leq d_i-g_i}$ to be the corresponding excursion, which is an element of the space of nonnegative càdlàg excursions $E=\{ e \in D(\R_+,\R_+): e(0)=0 \text{ and } \zeta(e)=\sup\{s: e(s) >0\}<+\infty\}$. Since $\cI$ is a local time at $0$ of the Markov process $\cX+\cI$, Itô's excursion theory tells us that

\[ \sum_i \delta_{(\cI_{g_i}, e^i)} \]
is a Poisson point measure and its intensity measure may be written as $\dd t\N_{\beta}(\dd e)$ where the Itô's excursion measure $\N_{\beta}$ is $\sigma$-finite on $E$ but not finite. With a somewhat abusive notation, we will denote by $\cX^{exc}$ an excursion distributed according to $\N_{\beta}$, and write $\cH^{exc}$ and $\cS^{exc}$ for the corresponding excursions in $\cH$ and $\cS$.

Finally, we recall that thanks to the scaling property of $\cX$ we can also define a family of probability measures $(\N_{\beta}^{(\sigma)})_{\sigma > 0}$ on $E$, where $\N_{\beta}^{(\sigma)}$ may informally be understood as $\N_{\beta}(\cdot \vert \zeta=\sigma)$ despite the fact that $\N_{\beta}(\zeta=\sigma)=0$. This family satisfies, for every measurable set $A \subset E$,

\[\N_{\beta}(A)=\int_0^\infty\frac{\N_{\beta}^{(\sigma)}(A)}{\beta\Gamma(1-1/\beta)\sigma^{1/\beta+1}}\dd\sigma.\]

In particular, we say that $\N_{\beta}^{(1)}$ is the law of an excursion of $\cX$ with duration $\zeta=1$, also called a normalized excursion of $\cX$. Again, we let $\cX^{(1)}$ denote such normalized excursion, $\cH^{(1)}$ denote the associated height process and $\cS^{(1)}$ denote the Brownian snake driven by $\cH^{(1)}$. Note that by excursion theory, $\cX^{(1)}$ may be sampled as follows: given the process $\cX$, set 
\[g_1=\sup\{s\leq1:\cX_s=\inf_{t\in[0,1]}\cX_t\}\quad\text{ and }\quad d_1=\inf\{s\geq1:\cX_s=\inf_{t\in[0,1]}\cX_t\}.\]
Then $\left((d_1-g_1)^{-1/\beta}(X_{g_1+t(d_1-g_1)}-X_{d_1})\right)_{t \geq 0}$ is a random excursion distributed according to $\N_{\beta}^{(1)}$.

\paragraph{Convergence of encoding processes.}

The first scaling limit results in the $\beta$-stable setting were obtained for the contour process alone, in order to study trees without spatial labelling. Duquesne and Le Gall \cite{DuquesneLeGall2002RandomTreesLevy} proved that the (properly rescaled) contour of a forest of i.i.d. Bienaymé trees converge in distribution to $\cH$, and then Duquesne \cite{Duquesne2003LimitTheoremContour} obtained that for a single Bienaymé tree conditioned to have a large size $n$, the contour process converges to $\cH^{(1)}$. 
Building on this, Marzouk \cite{Marzouk2020ScalingLimitsDiscrete} proved that the contour and the spatial contour of a snake indexed by such a large Bienaymé tree converge jointly toward $(\cH^{(1)},\cS^{(1)})$ when the jump distribution is sufficiently light (namely, when Assumption~\ref{assum:JumpTailAssumption} holds). Metz-Donnadieu \cite{TheseAlexis} has recently adapted Marzouk's argument to a forest of i.i.d. snakes, and here we will make use of the following slightly modified version where we consider an infinite sequence of i.i.d. $(\pi,\theta)$-random snakes $(W_i)_{i\geq 1}$.

Given this sequence $(W_i)_{i\geq 1}$, we let $C_i$ (resp $C^{sp}_i$) denote the contour (resp. spatial contour) process of $W_i$, and we let $C$ (resp. $C^{sp}$) be the process obtained by concatenating all $C_i$ together (resp. all $C^{sp}_i$ together). More precisely, we define $\zeta_0=-1$ and for $k \geq 1, \zeta_k = (k-1)+\sum_{i=1}^{k} \zeta(C_i)$, where $\zeta(C_i)$ is the length of $C_i$, i.e $\zeta(C_i)/2$ is the number of edges in $W_i$. Then we set $I_j=\inf\{ k \geq 0: j \leq \zeta_k \}$, and both $C$ and $C^{sp}$ are such that for $j \in\N$ with $I_j=i$, $C(j)=C_i(j-(\zeta_{i-1}+1))$ and  $C^{sp}(j)=C^{sp}_i(j-(\zeta_{i-1}+1))$.

\begin{theorem}\label{thm:ConvergenceSnakesInfiniteVolume}
Assume that $\pi \in \mathrm{dom}(\beta)$ and the jump distribution $\theta$ is centered and satisfies Assumption~\ref{assum:JumpTailAssumption}. Extend the discrete process $(I_j,C(j),C^{sp}(j))_{j \in\N}$ by linear interpolation to see it as an element of $\cC(\R_+,\R^3)$. Then the following convergence in distribution holds in the space $\cC(\R_+,\R^3)$ with respect to uniform convergence on compact subsets:
\begin{equation*}
    \left( \frac{I_{2nt} }{n^{1/\beta}\ell_{\pi}(n)}, \frac{\ell_{\pi}(n)}{n^{1-1/\beta}}C(2nt), \frac{1}{\sigma_{\theta}}\sqrt{\frac{\ell_{\pi}(n)}{n^{1-1/\beta}}}C^{sp}(2nt) \right)_{t \geq 0} \xrightarrow[n \to +\infty]{(d)} (\cI_t,\cH_t, \cS_t)_{t \geq 0},
\end{equation*}
where $\ell_{\pi}$ is the slowly varying function from \eqref{eq:DefAttractionStablebis}.
\end{theorem}

\begin{rmk}[Discussion of the proof]
If one forgets about the spatial contour process, the joint convergence of the first two processes is a corollary of \cite[Section~2]{DuquesneLeGall2002RandomTreesLevy}. Knowing this, it is straightforward to deduce the joint finite-dimensional convergence of those three processes, and the remaining point to discuss is the tightness of the spatial contour process. Marzouk provided an argument to obtain this tightness in the case of the single snake with an underlying tree conditioned to have $n$ edges, and Metz-Donnadieu adapted this argument in the case of a forest of $a_n$ i.i.d. snakes, where $a_n$ is such that $\frac{\zeta_{a_n}}{2n}$ converges in distribution to $\kappa_1 = \inf\{ t\geq 0: \cI_t \geq 1 \}$. In our framework, the infinite sequence of i.i.d. snakes can be seen as a sequence of i.i.d. forests consisting of $a_n$ snakes, and then $(C,C^{sp})$ simply is the concatenation of the i.i.d. contour and spatial contour processes of those forests. Thanks to Metz-Donnadieu's result, the sequence of i.i.d. spatial contour processes together with their durations converge in distribution, in particular it is tight (in the space $\{ (f,\zeta) \in \cC_c(\R_+,\R)\times\R_+: \forall t \geq \zeta, f(t)=0  \}^{\N}$ equipped with the product distance). Since in addition the cumulative durations of the first $m$ forests converge to  $\kappa_m = \inf\{ t\geq 0: \cI_t \geq m \}$ and $\kappa_m \to +\infty$ in probability as $m\to +\infty$, we deduce the tightness of the whole process $C^{sp}$ from the fact that concatenating $m$ continuous functions with compact support  is a continuous operation.
\end{rmk}

Theorem~\ref{thm:ConvergenceSnakesInfiniteVolume} should be useful to obtain scaling limits of an i.i.d. forest of snakes jointly with their subordinate trees. In addition, by following the method exposed in the proof of \cite[Proposition 2.5.2]{Duquesne2003LimitTheoremContour}, we can also use it to obtain limit theorems for a single snake $W$ conditioned to be large in a non-degenerate sense. In particular, we are interested in the conditioning where $\max W=\max C^{sp}$ is at least $\sigma_{\theta}n$, to obtain Proposition~\ref{prop:ConvergenceSnakesCondMax}. This is non-degenerate in the sense that for the limit we have $\N_{\beta}(\max\cS^{exc} \geq 1) \not\in \{0,+\infty\}$, due to the fact that $\P(\max_{[0,1]} \cS \geq 1) \not\in \{0,1\}$. 

Before proving Proposition~\ref{prop:ConvergenceSnakesCondMax}, recall that this result requires to change the time-scale by considering $k_n$ defined by \eqref{eq:time:change}, which is

\begin{equation*}
    \sqrt{\frac{\ell_{\pi}(k_n)}{k_n^{1-1/\beta}}} \sim \frac1n.
\end{equation*}

Indeed, given that $\max W \geq \sigma_{\theta}n$, the spatial contour will be of order $n$, thus in sight of Theorem~\ref{thm:ConvergenceSnakesInfiniteVolume} the right time-scale is $k_n$.

\begin{proof}[Proof of Proposition~\ref{prop:ConvergenceSnakesCondMax}]
As said above, we follow the method used for \cite[Proposition 2.5.2]{Duquesne2003LimitTheoremContour} to deduce this result from Theorem~\ref{thm:ConvergenceSnakesInfiniteVolume}. We consider the processes $(C_i,C^{sp}_i)_{i \geq 1}, (C, C^{sp}), (\zeta_i)_{i \geq 1}, I$ associated with the sequence of i.i.d. random snakes $(W_i)_{i \geq 1}$ involved in Theorem~\ref{thm:ConvergenceSnakesInfiniteVolume}. For all $n \in\N$, we also set
\[i_n:=\inf\{i \geq 1:\max C_i^{sp} \geq \sigma_{\theta}n\}.\]
Since $i_n$ is the index of the first snake that reaches the label $n$, we have $(C_{i_n},C_{i_n}^{sp}) \stackrel{(d)}{=} (C_{[n]},C_{[n]}^{sp})$. Moreover, $(C_{i_n},C_{i_n}^{sp})$ is the excursion of $(C, C^{sp})$ on the interval $[\zeta_{i_n-1}+1,\zeta_{i_n}]$.

Finally, we also consider the continuous processes $(\cI,\cH,\cS)$ and we define the times
\begin{equation*}
   \begin{aligned}
   g^{[\geq1]} &= \sup\{s \geq 0: \cH_s =0 \text{ and } \max_{[0,s]}\cS < 1 \}, \\
   d^{[\geq1]} &= \inf\{s \geq 0: \cH_s =0 \text{ and } \max_{[0,s]}\cS \geq 1, \}
   \end{aligned}
\end{equation*}
such that by excursion theory we may sample the variable $\left(\cH^{[\geq1]}_t,\cS^{[\geq1]}_t\right)_{t}$ as the excursion of $(\cH,\cS)$ on the interval $[g^{[\geq1]},d^{[\geq1]}]$. We thus want to prove that 

\begin{equation}\label{eq:excCV}
    \left( \frac{1}{n^2}C_{i_n}(2k_n t), \frac{1}{n \sigma_{\theta}}C_{i_n}^{sp}(2k_n t) \right)_{t \geq 0}\xrightarrow[n\rightarrow\infty]{(d)}\left(\cH^{[\geq1]}_t,\cS^{[\geq1]}_t\right)_{t\geq0}.
\end{equation}

The convergence of those excursions will follow from the convergence of the entire process and the convergence of the excursion intervals. First, by Skorokhod representation theorem, we can assume that the conclusion of Theorem~\ref{thm:ConvergenceSnakesInfiniteVolume} holds almost surely, and then extract along the subsequence $(k_n)_n$ to get
\begin{equation}\label{eq:convergence:as}
    \left( \frac{n^2 }{k_n} I_{2k_n t}, \frac{1}{n^2}C(2k_n t), \frac{1}{n \sigma_{\theta}}C^{sp}(2k_n t) \right)_{t \geq 0} \xrightarrow[n \to +\infty]{a.s.} (\cI_t,\cH_t, \cS_t)_{t \geq 0}.
\end{equation}

Then, since almost surely $\max_{[0,d^{[\geq1]}]} \cS > 1$, $\max_{[0,g^{[\geq1]}]} \cS < 1$ and $\cH$ is strictly positive on $(g^{[\geq1]},d^{[\geq1]})$, the convergences of the contour and spatial contour easily give 
\begin{equation*}
\begin{aligned}
     \limsup \frac{\zeta_{i_n-1}}{2k_n} &\leq g^{[\geq1]}, \\
     \liminf \frac{\zeta_{i_n}}{2k_n} &\geq d^{[\geq1]}. \\
\end{aligned}
\end{equation*}

There is only one way not to have an actual convergence in the previous display: there must be $s >0$ such that $\cH_s=0$ but which cannot be approximated by $\frac{\zeta_{j_n}}{2k_n}$ for some sequence $(j_n)_n$. However, recall that $(\frac{n^2 }{k_n} I_{2k_n t})_{t \geq 0}$ only increases on intervals of the form $[\frac{\zeta_{i}}{2k_n},\frac{\zeta_{i}+1}{2k_n}]$ and converges to $\cI$ whose increase points are exactly $\{s\geq 0: \cH_s=0 \}$. Because of this, we actually have $\frac{\zeta_{i_n-1}}{2k_n} \to g^{[\geq1]}$ and $\frac{\zeta_{i_n}}{2k_n} \to d^{[\geq1]}$, and combining this with \eqref{eq:convergence:as} implies \eqref{eq:excCV} as desired. At the same time, it shows that we also have
\begin{equation*}
    \frac{\# T_{i_n}}{k_n} = \frac{\zeta_{i_n}-\zeta_{i_{n-1}}-1}{2k_n} \xrightarrow[n \to +\infty]{a.s.} d^{[\geq1]}-g^{[\geq1]} = \zeta\bigl( \cH^{[\geq1]} \bigr).
\end{equation*}
This proves the second statement of Proposition~\ref{prop:ConvergenceSnakesCondMax}
\end{proof}

\section{Regularly varying functions}\label{sct:regular}

We give here a brief account of the properties of regularly varying
functions, and we refer to \cite{BinghamGoldieTeugels1987RegularVariation}
for a complete survey on those functions.

\begin{definition}\label{Regular:def}
A measurable function \(f:(0,+\infty)\mapsto (0,+\infty)\) is
said to be \begin{itemize}
    \item \emph{slowly varying} at \(+\infty\) when
\[\forall \lambda>0,\ \frac{f(\lambda x)}{f(x)} \xrightarrow[x \to \infty]{} 1;
\]
    \item \emph{regularly varying with index \(\beta \in \mathbb{R}\)} (or
\emph{\(\beta\)-varying}) at \(+\infty\) when
\[\forall \lambda>0,\ \frac{f(\lambda x)}{f(x)} \xrightarrow[x \to \infty]{} \lambda^{\beta},
\] or equivalently when \(x \mapsto x^{-\beta}f(x)\) is a slowly varying
function at \(+\infty\).
\end{itemize}  
\end{definition}

\begin{rmk}
We also say that a function $f$ defined on a neighbourhood of $a \in [0,+\infty)$ is regularly varying at $a^{+}$ (resp. $a^{-}$, when $a \neq 0$) if there is $g$ a regularly varying function at $+\infty$ such that $f(a+x^{-1}) \sim g(x)$ (resp. $f(a-x^{-1}) \sim g(x)$) as $x \to +\infty$.
Here by default a regularly varying function will be regularly varying at $+\infty$, and we will only write down explicitly the limiting point when it is more convenient to deal with regularly varying functions at $a \neq +\infty$.
\end{rmk}

The limiting power function in Definition~\ref{Regular:def}
may seem arbitrary, but the Characterization Theorem written below
states that as soon as the ratio \(\frac{f(\lambda x)}{f(x)}\)
converges, for a sufficiently large set of \(\lambda\), to some unknown
limiting function, then actually it converges for all \(\lambda>0\) to a
power function.

\begin{theorem}[{Characterization Theorem, \cite[Theorem~1.4.1]{BinghamGoldieTeugels1987RegularVariation}}]\label{Regular:Charac}
Consider a measurable function \(f:(0,+\infty) \to (0,+\infty)\). If
there is a set \(\Lambda \subset \mathbb{R_{+}}\) of positive Lebesgue
measure such that for all \(\lambda \in \Lambda\),
\(\frac{f(\lambda x)}{f(x)}\) converges to a limit in $(0,+\infty)$ as
\(x \to +\infty\), then \(f\) is regularly varying with some index
\(\beta \in \mathbb{R}\).
\end{theorem}

An obvious example of a \(\beta\)-varying function is
\(x \mapsto x^{\beta}\), and actually any regularly varying function
should be thought as roughly asymptotic to a power function (and any
slowly varying function as roughly constant) since many asymptotic
properties of power functions extend to regularly varying functions. In
this direction, the following result enables to bound regularly varying
functions by power functions.

\begin{prop}[{Potter's bound, 
\cite[Theorem 1.5.6]{BinghamGoldieTeugels1987RegularVariation}}]\label{Regular:Potter}
 Let \(f\) be a
\(\beta\)-varying function. For all \(C>1\) and \(\varepsilon>0\), there
is \(A>0\) such that
\[ \forall x,y \geq A,\ \frac{f(x)}{f(y)} \leq C\times\left( \frac{x}{y} \right)^{\beta+\varepsilon} \vee \left( \frac{x}{y} \right)^{\beta-\varepsilon}.
\] In particular, for all \(\varepsilon>0\) we have
\(x^{\beta-\varepsilon} \ll f(x) \ll x^{\beta+\varepsilon}\).
\end{prop}

Given two regularly varying functions \(f\) and \(g\) such that
\(g(x) \to +\infty\), \(f \circ g\) is clearly regularly varying at
\(+\infty\) with
\(\text{index}(f \circ g) = \text{index}(f) \times\text{index}(g)\).
This stability by composition also comes with the existence of
(asymptotic) inverse function for regularly varying function with
positive index.

\begin{prop}[{Asymptotic inversion, 
\cite[Theorem 1.5.12]{BinghamGoldieTeugels1987RegularVariation}}]\label{Regular:Inversion}
Let \(f\) be a regularly
varying function with index \(\beta >0\), then there is a
\(\frac{1}{\beta}\)-varying function \(g\) such that
\[ f(g(x)) \sim g(f(x)) \sim x \ \text{ as } x \to +\infty.
\] Moreover, \(g\) is unique up to asymptotic equivalence.
\end{prop}

An asymptotic inverse of a function \(f\) can be expressed by means
of the generalized inverse of monotone functions, but also by means of
the conjugate of slowly varying function, see
\cite{BinghamGoldieTeugels1987RegularVariation} for more details on
this.

Finally, the following results show that regularly varying functions can
be integrated, and sometimes differentiated, just as power functions. We write it in integral form but the exact same results hold for series.

\begin{theorem}[{Direct half of Karamata's Theorem, 
\cite[section 1.6]{BinghamGoldieTeugels1987RegularVariation}}]\label{Regular:Karamata}
Let \(\ell\) be a slowly
varying function. There is \(A>0\) such that \(\ell\) is locally bounded
on \([A,+\infty)\), hence we can integrate and we have the following
estimates:\begin{itemize}
    \item For \(\beta > -1\),
\[ \int_{A}^x t^{\beta}\ell(t)dt \sim_{x \to +\infty} \frac{x^{\beta+1}}{\beta+1} \ell(x); \]
    \item For \(\beta < -1\),
\[ \int_{x}^{+\infty} t^{\beta}\ell(t) dt \sim_{x \to +\infty} \frac{x^{\beta+1}}{-\beta-1}\ell(x);\]
    \item For \(\beta=-1\), the function
\(x \mapsto \int_{A}^x \frac{\ell(t)}{t}dt\) is slowly varying and such
that
\[ \frac{1}{\ell(x)}\int_{A}^x \frac{\ell(t)}{t}dt \xrightarrow[x \to +\infty]{} 0. 
\] Moreover if \(\int_{A}^{+\infty} \frac{\ell(t)}{t}dt <+\infty\) then
the function \(x \mapsto \int_{x}^{+\infty} \frac{\ell(t)}{t} dt\) is
also slowly varying and such that
\[ \frac{1}{\ell(x)}\int_{x}^{+\infty} \frac{\ell(t)}{t}dt \xrightarrow[x \to +\infty]{} +\infty.
\]
\end{itemize}
\end{theorem}

\begin{theorem}[{Monotone Density Theorem, 
\cite[Theorem 1.7.2]{BinghamGoldieTeugels1987RegularVariation}}]\label{Regular:MonotoneDensity}
Let \(f\) be a positive
measurable function, locally bounded on \([A,+\infty)\). Assume that
\(f\) is ultimately monotone.
\begin{itemize}
    \item If there is \(C>0\), \(\beta>0\) and a
slowly varying function \(\ell\) such that
\[ \int_{A}^{x}f(t)dt \sim_{x \to +\infty} C x^{\beta} \ell(x),
\] then \(f(x) \sim_{x \to +\infty} \beta C x^{\beta-1} \ell(x)\).
    \item If \(\int_{A}^{+\infty}f(t)dt < +\infty\) and in addition there is \(C>0\),
\(\beta < 0\) and a slowly varying function \(\ell\) such that
\[ \int_{x}^{+\infty} f(t)dt \sim_{x \to +\infty} C x^{\beta} \ell(x),
\] then \(f(x) \sim_{x \to +\infty} -\beta C x^{\beta-1} \ell(x)\).
\end{itemize}
\end{theorem}

\newpage

\section*{Tool and computational resource disclosure\footnote{See the recommendations of the Leiden Declaration on Artificial Intelligence and Mathematics (\url{https://leidendeclaration.ai/}).}}
No LLM or other automated tools falling into the current denomination of generative AI have been used during the writing process or the research process for this work.

\printbibliography

@article{Marckert2008LineageProcessGaltonWatson,
  title = {The Lineage Process in {{Galton-Watson}} Trees and Globally Centered Discrete Snakes},
  author = {Marckert, Jean-François},
  date = {2008},
  journaltitle = {Annals of Applied Probability},
  shortjournal = {Ann. Appl. Probab.},
  volume = {18},
  number = {1},
  pages = {209--244},
  issn = {1050-5164},
  doi = {10.1214/07-AAP450},
  url = {https://doi.org/10.1214/07-AAP450},
  fjournal = {The Annals of Applied Probability},
  langid = {english},
  mrnumber = {2380897},
  zbmath = {5251137},
  zmnumber = {1140.60042}
}

@article{MarckertMokkademStatesSpacesSnake,
  title = {States Spaces of the Snake and of Its Tour – {{Convergence}} of the Discrete Snake},
  author = {Marckert, Jean-Francois and Mokkadem, Abdelkader},
  doi = {10.1023/B:JOTP.0000012004.73051.f3},
  langid = {english}
}

@article{RieraRosalezexcursions,
 author = {Riera, Armand and Rosales-Ortiz, Alejandro},
 title = {The structure of the local time of {Markov} processes indexed by {L{\'e}vy} trees},
 journal = {Probab. Theory Relat. Fields},
 issn = {0178-8051},
 volume = {189},
 number = {1-2},
 pages = {1--99},
 year = {2024},
 language = {English},
}

@article{uniqueness,
author = {Jean-Fran{\c{c}}ois Le Gall},
title = {{Uniqueness and universality of the Brownian map}},
volume = {41},
journal = {The Annals of Probability},
number = {4},
publisher = {Institute of Mathematical Statistics},
pages = {2880 -- 2960},
year = {2013},
doi = {10.1214/12-AOP792},
URL = {https://doi.org/10.1214/12-AOP792}
}

@article{convergence,
author = {Gr{\'e}gory Miermont},
title = {{The Brownian map is the scaling limit of uniform random plane quadrangulations}},
volume = {210},
journal = {Acta Mathematica},
number = {2},
publisher = {Institut Mittag-Leffler},
pages = {319 -- 401},
year = {2013},
doi = {10.1007/s11511-013-0096-8},
URL = {https://doi.org/10.1007/s11511-013-0096-8}
}

@article{legallsurvey2019,
 author = {Le Gall, Jean-Fran{\c{c}}ois},
 title = {Brownian geometry},
 fjournal = {Japanese Journal of Mathematics. 3rd Series},
 journal = {Jpn. J. Math. (3)},
 issn = {0289-2316},
 volume = {14},
 number = {2},
 pages = {135--174},
 year = {2019},
}

@article{LeGall2005RandomTreesApplications,
  title = {Random Trees and Applications},
  author = {Le Gall, Jean-François},
  date = {2005},
  journaltitle = {Probability Surveys},
  shortjournal = {Probab. Surv.},
  volume = {2},
  pages = {245--311},
  issn = {1549-5787},
  doi = {10.1214/154957805100000140},
  url = {https://eudml.org/doc/229216},
  fjournal = {Probability Surveys},
  langid = {english},
  mrnumber = {2203728},
  zbmath = {5728580},
  zmnumber = {1189.60161}
}

@article{Neveu1986,
author = {Neveu, J.},
journal = {Annales de l'I.H.P. Probabilités et statistiques},
language = {fre},
number = {2},
pages = {199-207},
publisher = {Gauthier-Villars},
title = {Arbres et processus de Galton-Watson},
url = {http://eudml.org/doc/77276},
volume = {22},
year = {1986},
}

@book{Feller1971IntroductionProbabilityTheory,
  title = {An Introduction to Probability Theory and Its Applications},
  author = {Feller, William},
  date = {1971},
  series = {Wiley Series in Probability and Mathematical Statistics},
  edition = {Second ed},
  volume = {2},
  publisher = {J. Wiley},
  location = {New York Chichester Brisbane [etc.]},
  isbn = {978-0-471-25709-7},
  langid = {english}
}

@book{PitmanPicard2006CombinatorialStochasticProcesses,
  title = {Combinatorial Stochastic Processes: {{Ecole}} d'été de Probabilités de {{Saint-Flour XXXII}} - 2002},
  shorttitle = {Combinatorial Stochastic Processes},
  author = {Pitman, Jim and Picard, Jean},
  namea = {{Ecole d'été de probabilités de Saint-Flour}},
  nameatype = {collaborator},
  date = {2006},
  series = {Ecole d'été de Probabilités de {{Saint-Flour}}},
  number = {32},
  publisher = {Springer},
  location = {Berlin},
  isbn = {978-3-540-30990-1},
  langid = {english}
}

@article{Kortchemski2017SubexponentialTailBounds,
  title = {Sub-Exponential Tail Bounds for Conditioned Stable {{Bienaymé-Galton-Watson}} Trees},
  author = {Kortchemski, Igor},
  date = {2017},
  journaltitle = {Probab. Theory Relat. Fields},
  volume = {168},
  number = {1--2},
  pages = {1--40},
  issn = {0178-8051},
  doi = {10.1007/s00440-016-0704-6},
  url = {https://doi.org/10.1007/s00440-016-0704-6},
  fjournal = {Probability Theory and Related Fields},
  langid = {english},
  mrnumber = {3651047},
  zbmath = {6734667},
  zmnumber = {1374.60167}
}

@online{AlsmeyerHoang2025PowerfractionalDistributionsBranching,
  title = {Power-Fractional Distributions and Branching Processes},
  author = {Alsmeyer, Gerold and Hoang, Viet Hung},
  date = {2025},
  doi = {10.48550/ARXIV.2503.18563},
  url = {https://arxiv.org/abs/2503.18563},
  urldate = {2026-05-18},
  pubstate = {prepublished},
  version = {2}
}

@inproceedings{Aldous1991ContinuumRandomTreea,
  title = {The Continuum Random Tree. {{II}}. {{An}} Overview},
  booktitle = {Stochastic Analysis ({{Durham}}, 1990)},
  author = {Aldous, David},
  date = {1991},
  series = {London Math. {{Soc}}. {{Lecture}} Note Ser.},
  volume = {167},
  pages = {23--70},
  publisher = {Cambridge Univ. Press, Cambridge},
  doi = {10.1017/CBO9780511662980.003},
  url = {https://doi.org/10.1017/CBO9780511662980.003},
  isbn = {0-521-42533-6},
  mrclass = {60C05 (05C80)},
  mrnumber = {1166406},
  mrreviewer = {RobinPemantle}
}

@article{LeGallLin2016RangeTreeindexedRandom,
  title = {The Range of Tree-Indexed Random Walk},
  author = {Le Gall, Jean-François and Lin, Shen},
  date = {2016},
  journaltitle = {Journal of The Institute of Mathematics of Jussieu},
  shortjournal = {J. Inst. Math. Jussieu},
  volume = {15},
  number = {2},
  pages = {271--317},
  issn = {1474-7480},
  doi = {10.1017/S1474748014000280},
  url = {https://doi.org/10.1017/S1474748014000280},
  fjournal = {Journal of the Institute of Mathematics of Jussieu},
  langid = {english},
  mrnumber = {3480967},
  zbmath = {6568866},
  zmnumber = {1339.60049}
}

@article{Duquesne2003LimitTheoremContour,
  title = {A Limit Theorem for the Contour Process of Conditioned {{Galton-Watson}} Trees},
  author = {Duquesne, Thomas},
  date = {2003},
  journaltitle = {Annals of Probability},
  shortjournal = {Ann. Probab.},
  volume = {31},
  number = {2},
  pages = {996--1027},
  issn = {0091-1798,2168-894X},
  doi = {10.1214/aop/1048516543},
  url = {https://doi.org/10.1214/aop/1048516543},
  fjournal = {The Annals of Probability},
  langid = {english},
  mrnumber = {1964956},
  zbmath = {1933038},
  zmnumber = {1025.60017}
}

@article{JansonMarckert2005ConvergenceDiscreteSnakes,
  title = {Convergence of {{Discrete Snakes}}},
  author = {Janson, Svante and Marckert, Jean-François},
  date = {2005-07},
  journaltitle = {Journal of Theoretical Probability},
  shortjournal = {J Theor Probab},
  volume = {18},
  number = {3},
  pages = {615--645},
  issn = {0894-9840, 1572-9230},
  doi = {10.1007/s10959-005-7252-9},
  url = {http://link.springer.com/10.1007/s10959-005-7252-9},
  urldate = {2025-11-04},
  langid = {english}
}

@article{Khanfir2025HortonStrahlerNumberGaltonWatson,
  title = {The {{Horton-Strahler}} Number of {{Galton-Watson}} Trees with Possibly Infinite Variance},
  author = {Khanfir, Robin},
  date = {2025-10-01},
  journaltitle = {The Annals of Applied Probability},
  shortjournal = {Ann. Appl. Probab.},
  volume = {35},
  number = {5},
  eprint = {2307.05983},
  eprinttype = {arXiv},
  eprintclass = {math},
  issn = {1050-5164},
  doi = {10.1214/25-AAP2204},
  url = {http://arxiv.org/abs/2307.05983},
  urldate = {2025-11-20},
  langid = {english}
}

@article{LalleyShao2015MaximalDisplacementCritical,
  title = {On the Maximal Displacement of Critical Branching Random Walk},
  author = {Lalley, Steven P. and Shao, Yuan},
  date = {2015-06-01},
  journaltitle = {Probability Theory and Related Fields},
  shortjournal = {Probab. Theory Relat. Fields},
  volume = {162},
  number = {1},
  pages = {71--96},
  issn = {1432-2064},
  doi = {10.1007/s00440-014-0566-8},
  url = {https://doi.org/10.1007/s00440-014-0566-8},
  urldate = {2025-12-01},
  langid = {english}
}

@article{Maillard2016MaximumTreeindexedRandom,
  title = {The Maximum of a Tree-Indexed Random Walk in the Big Jump Domain},
  author = {Maillard, Pascal},
  date = {2016},
  journaltitle = {ALEA, Lat. Am. J. Probab. Math. Stat.},
  volume = {13},
  number = {2},
  pages = {545--561},
  issn = {1980-0436},
  doi = {10.30757/ALEA.v13-22},
  fjournal = {ALEA. Latin American Journal of Probability and Mathematical Statistics},
  langid = {english},
  mrnumber = {3519759},
  zbmath = {6604207},
  zmnumber = {1345.60101}
}

@article{Slack1968BranchingProcessMean,
  title = {A Branching Process with Mean One and Possibly Infinite Variance},
  author = {Slack, R. S.},
  date = {1968},
  journaltitle = {Z. Wahrscheinlichkeitstheorie und Verw. Gebiete},
  volume = {9},
  pages = {139--145},
  doi = {10.1007/BF01851004},
  url = {https://doi.org/10.1007/BF01851004},
  fjournal = {Zeitschrift für Wahrscheinlichkeitstheorie und Verwandte Gebiete},
  mrclass = {60.67},
  mrnumber = {228077},
  mrreviewer = {BernardHarris},
  readstatus = {read}
}

@article{ArcherCarranceMenard2024PropertiesStableSnakes,
  title = {Some Properties of Stable Snakes},
  author = {Archer, Eleanor and Carrance, Ariane and M{\'e}nard, Laurent},
  year = 2024,
  month = mar,
  eprint = {2403.15275},
  primaryclass = {math.PR},
  publisher = {arXiv},
  doi = {10.48550/ARXIV.2403.15275},
  archiveprefix = {arXiv},
  copyright = {arXiv.org perpetual, non-exclusive license}
}

@book{BinghamGoldieTeugels1987RegularVariation,
  title = {Regular Variation},
  author = {Bingham, N. H. and Goldie, C. M. and Teugels, J. L.},
  year = 1987,
  series = {Encyclopedia of Mathematics and Its Applications},
  volume = {27},
  pages = {xx+491},
  publisher = {Cambridge University Press, Cambridge},
  doi = {10.1017/CBO9780511721434},
  isbn = {0-521-30787-2},
  mrclass = {26A12 (11K65 11N60 30-02 40E05 60-02 60Fxx)},
  mrnumber = {898871},
  mrreviewer = {R.A.Maller}
}

@misc{HouJiangRenSong2023TailProbabilityMaximal,
  title = {Tail Probability of Maximal Displacement in Critical Branching {{L\'evy}} Process with Stable Branching},
  author = {Hou, Haojie and Jiang, Yiyang and Ren, Yan-Xia and Song, Renming},
  year = 2023,
  month = oct,
  number = {arXiv:2310.05323},
  eprint = {2310.05323},
  primaryclass = {math},
  publisher = {arXiv},
  doi = {10.48550/arXiv.2310.05323},
  urldate = {2025-10-15},
  archiveprefix = {arXiv},
  langid = {english}
}

@book{DuquesneLeGall2002RandomTreesLevy,
  title = {Random {{Trees}}, {{Lévy Processes}} and {{Spatial Branching Processes}}},
  author = {Duquesne, Thomas and Le Gall, Jean-François},
  date = {2002},
  edition = {Astérisque},
  issn = {03031179, 24925926},
  doi = {10.24033/ast.545},
  url = {https://smf.emath.fr/publications/arbres-aleatoires-processus-de-levy-et-processus-de-branchement-spatiaux},
  urldate = {2026-03-27},
  langid = {english}
}

@article{LeGall2017SubordinationTreesBrownian,
  title = {Subordination of Trees and the {{Brownian}} Map},
  author = {Le Gall, Jean-Fran{\c c}ois},
  year = 2017,
  month = aug,
  journal = {Probability Theory and Related Fields},
  volume = {171},
  number = {3--4},
  pages = {819--864},
  publisher = {{Springer Science and Business Media LLC}},
  issn = {1432-2064},
  doi = {10.1007/s00440-017-0794-9}
}

@article{LeGallLin2015RangeTreeindexedRandom,
  title = {The Range of Tree-Indexed Random Walk in Low Dimensions},
  author = {Le Gall, Jean-Fran{\c c}ois and Lin, Shen},
  year = 2015,
  journal = {Annals of Probability},
  volume = {43},
  number = {5},
  pages = {2701--2728},
  issn = {0091-1798},
  doi = {10.1214/14-AOP947},
  fjournal = {The Annals of Probability},
  langid = {english},
  mrnumber = {3395472},
  zbmath = {6502685},
  zmnumber = {1335.60070}
}

@article{Marzouk2020ScalingLimitsDiscrete,
  title = {Scaling Limits of Discrete Snakes with Stable Branching},
  author = {Marzouk, Cyril},
  year = 2020,
  journal = {Ann. Inst. Henri Poincar\'e, Probab. Stat.},
  volume = {56},
  number = {1},
  pages = {502--523},
  issn = {0246-0203},
  doi = {10.1214/19-AIHP970},
  fjournal = {Annales de l'Institut Henri Poincar\'e. Probabilit\'es et Statistiques},
  langid = {english},
  mrnumber = {4058997},
  zbmath = {7199314},
  zmnumber = {1434.60248}
}

@online{Zhang2026EmptyBallsCritical,
  title = {On Empty Balls of Critical 2-Dimensional Branching Random Walks},
  author = {Zhang, Shuxiong},
  date = {2026-01-27},
  eprint = {2502.16225},
  eprinttype = {arXiv},
  eprintclass = {math.PR},
  doi = {10.48550/arXiv.2502.16225},
  url = {http://arxiv.org/abs/2502.16225},
  urldate = {2026-07-10}
}

@online{Addario-BerryBrandenbergerHamdanKerriou2022UniversalHeightWidth,
  title = {Universal Height and Width Bounds for Random Trees},
  author = {Addario-Berry, Louigi and Brandenberger, Anna and Hamdan, Jad and Kerriou, Céline},
  date = {2022-04-25},
  eprint = {2105.03195},
  eprinttype = {arXiv},
  eprintclass = {math.PR},
  doi = {10.48550/arXiv.2105.03195},
  url = {http://arxiv.org/abs/2105.03195},
  urldate = {2026-08-10},
  pubstate = {prepublished}
}

@online{article2,
title = {A bijective approach to study the metric net of random quadrangulations},
author = {Aurillard, Antoine and Mourichoux, Mathieu},
pubstate={inpreparation}
}

@phdthesis{TheseAlexis,
title = {Markov property for the vertical profile of random labelled trees and models of Brownian geometry.},
author = {Metz-Donnadieu, Alexis},
pubstate={inpreparation}
}

\end{document}